\documentclass[10pt,leqno]{amsart}
\usepackage{srcltx}
\usepackage{amsmath,amssymb,cases,color}
\usepackage{hyperref}
\usepackage{bm}
\usepackage[capitalise]{cleveref}
\usepackage{vmargin}
\usepackage{dsfont}
\usepackage{tikz}
\usepackage{amsmath, amssymb,amscd, }
\usepackage{amsfonts}
\usepackage{mathrsfs}
\usepackage{graphicx}
\usepackage[shortlabels]{enumitem}
\usepackage{upref}
\usepackage{nicefrac}
\hypersetup{linkcolor=blue, colorlinks=true,citecolor = red}
\newtheorem{theorem}{Theorem}[section]
\newtheorem{proposition}[theorem]{Proposition}
\newtheorem{lemma}[theorem]{Lemma}
\newtheorem{cor}[theorem]{Corollary}
\newtheorem{remark}[theorem]{Remark}

\theoremstyle{definition}

\theoremstyle{remark}
\numberwithin{equation}{section}

\newcommand{\rH}{\mathrm{H}}
\newcommand{\rB}{\mathrm{B}}
\newcommand{\rC}{\mathrm{C}}
\newcommand{\rL}{\mathrm{L}}
\newcommand{\rW}{\mathrm{W}}
\newcommand{\E}{\mathbb{E}}
\newcommand{\mb}{\mathbb}

\newcommand{\R}{\mathbb{R}}
\newcommand{\N}{\mathbb{N}}

\newcommand{\D}{\operatorname{div}}

\renewcommand{\leq}{\leqslant}
\renewcommand{\geq}{\geqslant}


\begin{document}

	\title[Fluid-Porous]{Fluid-Structure Interaction with Porous Media: The Beaver-Joseph condition in the strong sense}
	
%	\date{\today}
	
	\author{Tim  Binz} 
	\address{T. Binz, Princeton University,
		Program in Applied \& Computational Mathematics, 
		Fine Hall, Washington Road, 08544 Princeton, NJ, USA.}
%		Technische Universit\"{a}t Darmstadt,
%		Schlo\ss{}gartenstra{\ss}e 7, 64289 Darmstadt, Germany.}
	\email {tb7523@princeton.edu}
	
	\author{Matthias Hieber} 
	\address{M. Hieber, Technische Universit\"{a}t Darmstadt,
		Schlo\ss{}gartenstra{\ss}e 7, 64289 Darmstadt, Germany.}
	\email {hieber@mathematik.tu-darmstadt.de}

	\author{Arnab Roy}
	\address{A. Roy, Technische Universit\"{a}t Darmstadt,
		Schlo\ss{}gartenstra{\ss}e 7, 64289 Darmstadt, Germany.}
	\email{aroy@bcamath.org}

\begin{abstract}
{\color{black}This article considers fluid structure interaction describing the motion of a fluid contained in a porous medium.} The fluid is modelled by Navier-Stokes equations and the coupling between fluid and the porous medium is described by the classical Beaver-Joseph or the Beaver-Joseph-Saffman interface condition. In contrast to previous work these conditions are investigated for the first time in the strong sense and it is shown that the coupled system admits a unique, global strong solution in critical spaces provided the data are small enough. Furthermore, a Serrin-type blow-up criterium is developed and higher regularity estimates at the interface are established, which say  that the solution is even analytic provided the forces are so.      
\end{abstract}

\maketitle

\section{Introduction}
\label{sec:intro}

The interaction between a fluid and a porous material is a classical problem which couples the Navier-Stokes equations for an incompressible fluid to the equations of 
saturated porous material. In the latter case, the filtration of a fluid through porous media is often described by Darcy's law. This law describes the simplest relation between the 
velocity and the pressure in porous media under the physical reasonable assumption that the associated fluid flows are small. Darcy's law for the motion of an incompressible fluid  through a 
saturated porous media is described by   
\begin{equation*}
	s\partial_t p + \operatorname{div} {u}_{p}=0,\quad {u}_{p}=- {K}\nabla p \quad \mbox{in}\quad (0,T)\times {\Omega}_p,
\end{equation*}
where $s\geq 0$ is a coefficient and  $K$ denotes the hydraulic conductivity tensor of the porous media. Here the underlying domain $\Omega$ is decomposed into two non-intersecting subdomains 
$\Omega_f$ and $\Omega_p$ separated by an interface $\Gamma$  where $\Gamma = \partial \Omega_f \cap \partial \Omega_p$.  Let $\Omega_f$ be the  fluid region with the flow governed by Navier--Stokes 
equations
\begin{equation*}\label{fluid:mom}
	\partial _{t}{u} + ({u} \cdot \nabla){u} -\mbox { div } {\sigma({u}, \pi)} =  0, \quad \operatorname{div}{u}=0\quad \mbox{in}\quad (0,T)\times {\Omega}_f,
\end{equation*}
where $u$ and $\pi$ denote the velocity and the pressure of the fluid in  ${\Omega}_f$, respectively. As usual, $\sigma(u,\pi)$ denotes the Cauchy stress tensor. 

Coupling conditions across the interface which separates fluid flow in $\Omega_f$ and the porous media in $\Omega_p$ are of central importance in this context. They have been investigated both 
from the physical and rigorous mathematical point of view, see e.g., the works of Beaver-Joseph \cite{BJ:67}, Saffman \cite{Saf:71} and J\"ager and Mikelic \cite{JM:00}. 
It is natural to consider two classical interface conditions on the fluid-porous media interface $\Gamma$, which read as
\begin{equation*}
	{u}\cdot {n}_f + {u}_p\cdot {n}_p =0 \quad   % \mbox{on}\quad (0,T)\times \Gamma_{fp},\label{inter1}\\
	-\sigma(u, \pi){n}_f\cdot {n}_f = p  \\ 
	%-\sigma_f( {u}_f, p_f){n}_f\cdot {\tau}_{f,j} = \alpha {K_j^{-1/2}}({u}_f-{u}_p)\cdot {\tau}_{f,j}, \ j=1,2, 
\end{equation*}
where ${n}_f$ and  ${n}_p$ denote the outward unit normal vectors to ${\Omega}_f$, ${\Omega}_p$ respectively,
These two conditions are classical: they describe that the  exchange of fluid between the two domains is conservative and that the kinematic pressure in the porous media balances the 
normal component of the fluid flow stress in normal direction. 

As a third and central  interface equation we consider  the   {\em Beaver--Joseph condition} \cite{BJ:67} which says that  the tangential component of the stress of the 
fluid flow is proportional to   the jump in the tangential velocity over the interface. More precisely, Beaver and Joseph \cite{BJ:67} proposed that the normal derivative of the 
tangential component of $u$ should satisfy       
\begin{align*}
	-\sigma( {u},\pi){n}_f\cdot {\tau}_{j} = \alpha {K_j^{-1/2}}({u}-{u}_p)\cdot {\tau}_{j}, \ j=1,2, 
\end{align*}
where ${\tau}_{1},{\tau}_{2}$ represent a local orthonormal basis for the tangent plane to $\Gamma$, $K_j= (K{\tau}_j)\cdot {\tau}_{j}$ and $\alpha$ is a friction coefficient.
The  system  is complemented by boundary conditions for $u$ and $p$ described precisely at the end of Section 2.   

For a rigorous derivation of the Beaver-Joseph  condition based on stochastic homogenization, we refer to the work of J\"ager and Mikelic \cite{JM:00}.   

Saffman \cite{Saf:71} pointed out that the velocity $u_p$ is much smaller than the other terms in the Beaver-Joseph law and proposed the simplified law 
\begin{align*}
	-\sigma( {u},\pi){n}_f\cdot {\tau}_{j} = \alpha {K_j^{-1/2}}u \cdot {\tau}_{j}, \ j=1,2, 
\end{align*}
which is nowadays called the {\em Beaver-Joseph-Saffman} condition.

{\color{black}
The forces interacting between the two phases give, in principle,  rise to time dependent domains and to moving interfaces. However, the modeling of the underlying  Beaver-Joseph coupling conditions
in the situation of moving domains is far from being fully understood and we hence  restrict ourselves here to the case of a fixed interface. The investigation of the case of moving interfaces
is subject to future work.         
}

The above Navier-Stokes-Darcy problem subject to Beaver-Joseph or Beaver-Joseph-Saffman interface conditions has received a lot of attention during the last decades both from the analytical and numerical 
point of view. It has been investigated by many authors in the weak setting and under 
various simplifications such as linear Stokes-Darcy and stationary Stokes-Darcy systems. We refer here to the works of J\"ager and Mikelic \cite{JM:00}, Badia, Quaini and Quarteroni \cite{BQQ:09}, 
Cao, Gunzburger, Wang  \cite{CGHW:10}, Wang, Wu \cite{WW:21} and Cao, Wang \cite{CW:22}. For related results concerning the linear Stokes flow with poroelastic structures, we refer to the 
articles \cite{KCMa:23},\cite{KCMb:23}  by Kuan, Canic and Muha and \cite{BGSW16} by  Bociu, Guidoboni, Sacco and Webster and \cite{BCMW:21} by Bociu, Canic, Muha and Webster.  
For numerical studies we refer e.g., to the survey articles \cite{DQ:09}. 

All of the results cited above have in common that they treat and understand the Beaver-Joseph or Beaver-Joseph-Saffman condition in the weak sense. Hence, existing  well-posedness results concern 
weak solution and uniqueness questions are left open until today.  
%Moreover, we want to mention that in the context of the interaction between an incompressible, viscous fluid modeled 
%by the dynamic \textit{Stokes equation and a poroelastic structure}, there are some exciting recent existence results on weak solutions \cite{bociu2021multilayered, kuan2023existence, kuan2023fluid} and 
%on strong solutions \cite{avalos2024weak}.  

It is the  aim of this article to establish {\em existence and uniqueness results} for  {\em strong solutions} to the above  Navier-Stokes-Darcy problem subject to Beaver-Joseph or Beavier-Joseph-Saffman
conditions. In particular, solutions given in Theorem \ref{thm:main bjs critical}  and Theorem \ref{cor:bj r=q=2} below satisfy the Beavers-Joseph-Saffman and Beaver-Joseph  boundary conditions in the 
sense of traces. 

Let us emphasize that  our main results on global strong well-posedness  to the Navier-Stokes-Darcy model subject to  small initial data and on the local existence and uniqueness for  
strong solutions are formulated  even in  {\em critical spaces}. Critical spaces for the Navier-Stokes equations have been considered by many authors during the last decades; here we make use of the
abstract theory of critical spaces developed in \cite{PSW18}.    
We show that these critical spaces coincide with Besov-spaces of the form $B^{3/q-1}_{q,r}$, as in the case of the classical 
Navier-Stokes equations. In the latter case, these critical spaces go back to the work of Cannone \cite{Can:97}, \cite{Can:04}. Our approach, however, is very different from the one 
given in \cite{Can:97} and \cite{Can:04} which were based on Littlewood-Paley decomposition and not allowing coupling conditions.

A main difficulty in proving global existence and uniqueness results in the situation of the Beavers-Joseph or the Beaver-Joseph-Saffman condition is to understand the boundary condition in the strong sense, 
i.e. in the  sense of traces. To this end, we rewrite first the Beaver-Joseph-Saffman condition as a condition for the domain of the operator $A$ coupling the two equations. Note that the domain of this 
$2 \times 2$-system is then not diagonal. We then employ  rather sophisticated boundary perturbation arguments to decouple the system and to show that $A$ generates an exponentially decaying, compact and 
analytic semigroup in the $L^q\times L^q_\sigma$-setting for $1<q<\infty$. 

In a second step, we show that the coupling operator $A$ is even $\mathcal{R}$-sectorial, hereby using perturbation methods  for $\mathcal{R}$-sectorial operators.  
This allows us then to use modern semilinear theory based on maximal regularity and to obtain existence and uniqueness of solutions even in  critical Besov spaces of  the form $B^{3/q-1}_{q,r}$, where
$q,r$ are satisfying 
the condition $2/3r + 1/q \leq 1$. Our result can be thus viewed as an analogue of results due to Kato \cite{Kato:84}, Giga \cite{Gig:86} and Cannone \cite{Can:97} in the situation of the 
classical Navier-Stokes equations.     

Starting from this situation  we obtain furthermore a  rather complete understanding of the  coupled Navier-Stokes-Darcy system subject to the Beaver-Joseph-Saffman interface condition. 
Indeed, we develop first a Serrin-type finite time {\em blow-up criterium} and applying Angenent's parameter trick \cite{Ang:90} we obtain higher regularity estimates at the interface saying that the 
solution is even {\em analytic} provided the forces are so. Here regularity is meant to be only in tangential direction  of $\Gamma$ and a technical condition of the form $1/r + 3/2q <1$ is required. 
Note that this condition rules out the Hilbert space setting. This is also one reason why we deal with the $L^r_t-L^q_x$-setting for $r,q \ne 2$.   

In the situation of the Beaver-Joseph interface  condition a generation result for an analytic semigroup in $L^{q}$-setting for $q \not= 2$, similarly to the one 
described above, remains an open problem.  Nevertheless, we show that the Navier-Stokes-Darcy system subject to the Beaver-Joseph interface conditions are strongly, globally well-posed within 
the $L^2$-framework provided the parameter $\beta$ is small enough. 

It is interesting to note that the modern semilinear theory involved yields that the  initial data belong to the scaling invariant 
{\em critical Lions-Magenes space}. It can be seen as an  analogue of the classical result of Fujita and Kato \cite{FK:64} for the classical Navier-Stokes equations. Note that due to the lack 
of  understanding of the Stokes-Beaver-Joseph semigroup in the $L^q$-setting for $q \ne 2$, iteration techniques as developed by  Kato \cite{Kato:84} or Giga \cite{Gig:86} cannot be used here.

\section{Description of the model: Beaver-Joseph and Beaver-Joseph-Saffman conditions}

We consider the following model for fluid-porous media interaction in a  domain $\Omega = \Omega_p {\cup} \Omega_f$, where $\Omega_p$ and $\Omega_f$ are the 
domains occupied by the porous media and the fluid, respectively. The two domains are separated by a common interface $\Gamma$. The region $\Omega_f$ is occupied by an incompressible Newtonian 
fluid, whereas the domain $\Omega_p$ is filled by a saturated porous structure. 
In the porous structure we distinguish between the skeleton, composed by solid material, void porous space, and the fluid phase that consists of the fluid filling the pores. The porous media 
is regarded as superposition of the skeleton and the fluid phase and is assumed to be an incompressible material. The flow in the porous region  $\Omega_p$ is governed by  
by the equations
\begin{equation*}
	s\partial_t p + \operatorname{div} {u}_{p}=f_p,\quad {u}_{p}=- {K}\nabla p \quad \mbox{in}\quad (0,T)\times {\Omega}_p,
\end{equation*}
which describe in the first equation the saturated flow model and in the second one, Darcy's law. Here  $s\geq 0$ denotes the storage coefficient and $K$ denotes the hydraulic conductivity tensor of 
the porous media. Combining these two equations we obtain 
\begin{equation}\label{poro:darcy}
	s \partial_t p +  \operatorname{div}(- K \nabla p) = f_p  \mbox{ in } \Omega_p. 
\end{equation}
We impose the  boundary conditions $p=0$ on $\Gamma_1$, where $\Gamma_1$ is defined as in Figure 1 below.  

%\begin{figure}\label{fig}
%\begin{center}
%\begin{tikzpicture}
%\draw (-1,0) ellipse (0.4 and 0.7);
%\draw (-1,0) ellipse (2 and 2.2);
%\draw  (-1,0) ellipse (1.5 and 1);
%\node [align=left] at (-2,0) {$\Omega_{p}$};
%\node [align=left] at (-0.4,0) {$\Gamma_{1}$};
%\node [align=left] at (-1,1.5) {$\Omega_{f}$};
%\node [align=left] at (0.65,0) {$\Gamma$};
%\node [align=left] at (1.3,0) {$\Gamma_{2}$};
%\end{tikzpicture}
%\caption{Fluid-porous media interaction: Prototype Geometry}
%\end{center}
%\end{figure}

The fluid contained in $\Omega_f$ is governed by the incompressible  Navier--Stokes equations
\begin{equation*}\label{fluid:mom}
	\partial_{t}{u} + ({u} \cdot \nabla){u} -\mbox { div }\sigma(u,\pi) =  0, \quad \operatorname{div}{u}=0 \quad \mbox{in}\quad (0,T)\times {\Omega}_f,
\end{equation*}
where $\sigma(u,\pi)$ denotes the Cauchy stress tensor given by $\sigma(u,\pi)= 2\mu D(u) - \pi \mathbb{I}$, and  $D(u)$ is the deformation  tensor. 

The natural transmission conditions at the interface of a fluid and an elastic solid are described by the continuity of velocities and stresses. We now describe coupling conditions 
between a fluid and a porous media in some detail for fixed, time-independent domains and the interface $\Gamma$.  Conservation of mass at the interface as well as continuity of pressure or vanishing 
tangential velocity are classical interface conditions.

Beaver and Joseph \cite{BJ:67} pointed out that a fluid in contact with a porous media flows faster along the interface than if it were in contact with a solid interface. Hence, there is a certain slip 
of the fluid at the interface with porous media. They proposed that the normal derivative of the tangential component of $u$ should satisfy      
\begin{align*}
	%{u}_f\cdot {n}_f + {u}_p\cdot {n}_p =0 \quad   % \mbox{on}\quad (0,T)\times \Gamma_{fp},\label{inter1}\\
	%-\sigma_f( u_f, p_f){n}_f\cdot {n}_f = p_p  \\ 
	-\sigma( {u},\pi){n}_f\cdot {t}^{j}_{\Gamma} = \alpha {K_j^{-1/2}}({u}-{u}_p)\cdot {t}^{j}_{\Gamma}, \ j=1,2, 
\end{align*}
where ${n}_f$, ${n}_p$ are the outward unit normal vectors to ${\Omega}_f$, ${\Omega}_p$ respectively, ${t}^{1}_{\Gamma},{t}^{2}_{\Gamma}$ represents a local orthonormal basis for the tangent plane 
to $\Gamma$, $K_j= (K{\tau}_{f})\cdot {t}^{j}_{\Gamma}$ and $\alpha$ is a friction coefficient. This condition is called the {\em Beaver-Joseph condition}.
One further requires the conditions
\begin{align*}
	{u}\cdot {n}_f + {u}_p\cdot {n}_p &=0 \quad    \mbox{on}\quad (0,T)\times \Gamma,\label{inter1}\\
	-\sigma( u,\pi){n}_f\cdot {n}_f &= p \quad    \mbox{on}\quad (0,T)\times \Gamma,
	%-\sigma_f( {u}_f, p_f){n}_f\cdot {\tau}_{f,j} = \alpha {K_j^{-1/2}}({u}_f-{u}_p)\cdot {\tau}_{f,j}, \ j=1,2, 
\end{align*}
The conditions are classical: the exchange of fluid between the two domains is conservative and the kinematic pressure in the porous media balances the normal component of the fluid 
flow stress in normal direction.

Saffman pointed out that the velocity $u_p$ is much smaller than the other terms in the Beaver-Joseph law and proposed the simplified condition 
\begin{align*}
	-\sigma({u},\pi){n}_f\cdot {t}^{j}_{\Gamma} = \alpha {K_j^{-1/2}}{u}\cdot {t}^{j}_{\Gamma}, \ j=1,2, 
\end{align*}
which is nowadays called the {\em Beaver-Joseph-Saffman condition}. 

Describing the above models in the language of PDEs, we obtain the following systems. Let $\Omega_p$, $\Omega_f$ and $\Omega_i$ be bounded domains in $\mathbf{\R}^3$ with smooth boundaries 
$\partial\Omega_p$, $\partial\Omega_f$ and $\Gamma_1:=\partial\Omega_i$, respectively. Assume that $\Omega_i \ne \emptyset$ such that $\Omega_i \subset \Omega_p$ and such that 
$\partial\Omega_p = \Gamma \cup \Gamma_1$ and $\Gamma_1 \cap \Gamma = \emptyset$. Furthermore, we assume that $\partial\Omega_f = \Gamma \cup \Gamma_2$ and $\Gamma_2 \cap \Gamma = \emptyset$.     

\begin{figure}
  \centering
  \includegraphics[width=0.55\textwidth]{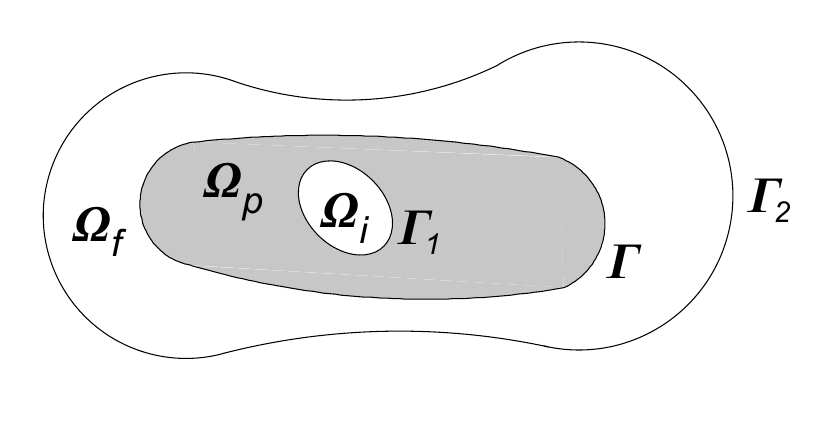}
  \caption{Fluid-porous media interaction}
  \label{fig1}
\end{figure}

Assuming $s=1$, $K=k\mathbb{I}$ in \eqref{poro:darcy}, the flow in the porous media region is governed by:
\begin{equation} \label{eq:ns-darcyp}
	\left\{
	\begin{aligned}
		\partial_t p - k \Delta p&=f_p &&\text{ in } (0,T)\times\Omega_p , \\ p&=0 &&\text{ on }(0,T)\times \Gamma_1,  \\ p(0)&=p_0 &&\text{ in } \Omega_p.
	\end{aligned}
	\right.
\end{equation}
The Navier-Stokes equations governing the flow in $\Omega_f$ are given by
\begin{equation} \label{eq:ns-darcyu}
	\left\{
	\begin{aligned}
		\partial_t u-\D\sigma&=-u \cdot \nabla u + {f}_f &&\text{ in } (0,T)\times\Omega_f,\\
		\operatorname{div}u&=0 &&\text{ in } (0,T)\times\Omega_f, \\ u &= 0 &&\text{ on } (0,T)\times\Gamma_2, %\\  u(0)&=u_0 &&\text{ in } \Omega_f.
	\end{aligned}
	\right.
\end{equation}
Here the Cauchy stress tensor for the fluid is given by
{\color{black}
\begin{equation*}
	\sigma = 2\mu D(u)- \pi \mathbb{I},
\end{equation*}
}
where the deformation tensor  is denoted by $D(u)=[D(u)_{ij}]_{1\leq i,j\leq 3}$ with $D(u)_{ij}=\frac{1}{2}(\partial_i u_j +\partial_j u_i)$. Moreover, $\mu>0$ denotes the viscosity of 
the fluid. 

The Beaver-Joseph-Saffman interface conditions read as 
\begin{equation} \label{eq:ns-darcyint}
	\left\{
	\begin{aligned}
		u\cdot n_{\Gamma}&=-k\nabla p\cdot n_{\Gamma} &&\text{ on }(0,T)\times\Gamma,\\
		n_{\Gamma}\cdot\sigma n_{\Gamma}&=-p &&\text{ on }(0,T)\times\Gamma,\\
		{t}^{j}_{\Gamma}\cdot\sigma n_{\Gamma}&=-\beta u\cdot {t}^{j}_{\Gamma},\ j=1,2,\ &&\text{ on }(0,T)\times\Gamma.
	\end{aligned}
	\right.
\end{equation}
Here $\beta = \alpha k^{-1/2}>0$ is a constant, $\{{t}^{1}_{\Gamma}, {t}^{2}_{\Gamma}\}$ represent the orthonormal basis for the tangent plane to $\Gamma$ and $n_{\Gamma}$ is the outward unit normal to $\Gamma$. 
The interface condition \eqref{eq:ns-darcyint}$_3$ is the Beavers-Joseph-Saffman condition which states that the tangential component of the normal stress of the fluid flow is 
proportional to the tangential velocity over the interface. We have seen that  condition \eqref{eq:ns-darcyint}$_3$ is a simplified variant of the original Beavers-Joseph interface condition, which reads as 
\begin{equation}\label{bdry:BJ}
	{t}^{j}_{\Gamma}\cdot\sigma n_{\Gamma}=-\beta(u- k\nabla p)\cdot {t}^{j}_{\Gamma},\ j=1,2,\ \text{ on }(0,T)\times\Gamma
\end{equation}
for some $\beta>0$. {\color{black} Note that the boundary conditions in \eqref{eq:ns-darcyint} and \eqref{bdry:BJ} are independent of a chosen local orthonormal basis $t^{j}_{\Gamma}$ for the
tangent plane to $\Gamma$, see \cite{BJ:67, Saf:71, JM:00} for details.}

\section{Preliminaries: semilinear evolution equations, function spaces, boundary perturbations, $\mathcal{R}$-sectorial operators}
\label{sec:prelim}

\textcolor{black}{In this section we collect various preliminariy results  needed in the sequel.}

\subsection{Semilinear evolution equations}
\

\textcolor{black}{We start by recalling  the theory of semilinear evolution equations based on the  maximal regularity approach in time-weighted spaces. For further information,  we refer to the
works of Pr\"{u}ss and Wilke \cite{PW:17} and Pr\"{u}ss, Simonett and Wilke \cite{PSW18}.}

We introduce first  time weighted spaces $\rL^r_\mu(J;E)$. More precisely, for a Banach space $(E,\|\cdot\|_E)$, $J = (0,T)$, $0 < T \le \infty$, $r \in (1,\infty)$, $\mu \in (\frac{1}{r},1]$ as well 
as $u \colon J \to E$, we denote by $t^{1-\mu} u$ the function $t \mapsto t^{1-\mu} u(t)$ on $J$. We then define
\begin{equation*}
	\rL_\mu^r(J;E) := \{u \colon J \to E : t^{1-\mu} u \in \rL^r(J;E)\}.
\end{equation*}
The latter space becomes a Banach space when equipped with the norm
\begin{equation*}
	\|u \|_{\rL_\mu^r(J;E)} := \big\| t^{1-\mu} u \big \|_{\rL^r(J;E)} = \Big(\int_J t^{r(1-\mu)} \| u(t) \|_E^r dt\Big)^{\frac{1}{r}}.
\end{equation*}
For $k \in \mathbb{N}_0$, the associated weighted Sobolev spaces are defined by
\begin{equation*}
	\rW_\mu^{k,r}(J;E) = \rH_\mu^{k,r}(J;E) := \left\{u \in \rW_\mathrm{loc}^{k,1}(J;E) : u^{(j)} \in \rL_\mu^r(J;E), \enspace j \in \{0,\dots,k\}\right\},
\end{equation*}
and these spaces are equipped with the norms
\begin{equation*}
	\| u \|_{\rW_\mu^{k,r}(J;E)} = \| u \|_{\rH_\mu^{k,r}(J;E)} := \Big(\sum_{j=0}^k \| u^{(j)} \|_{\rL_\mu^r(J;E)}^r \Big)^{\frac{1}{r}}.
\end{equation*}
The weighted fractional Sobolev spaces and Bessel potential spaces are then defined by interpolation.

Let $X_0, X_1$ be two Banach spaces such that $X_1 \hookrightarrow X_0$ is densely embedded. We denote by $X_\beta := [X_0,X_1]_{\beta}$ the complex interpolation spaces for $\beta \in (0,1)$ and by 
$X_{\theta,r} := (X_0,X_1)_{\theta,r}$ the real interpolation spaces for $\theta \in (0,1)$. Furthermore, the trace space is given by
\begin{equation*}
	X_{\gamma,\mu} := X_{\mu-\frac{1}{r},r} .
\end{equation*}
Moreover, the maximal $\rL^r$-regularity space is defined as
\begin{equation*}
	\E_{1,\mu}(J) := \rH^{1,r}_\mu(J;X_0) \cap \rL^r_\mu(J;X_1)  
\end{equation*}
and the data space as
\begin{equation*}
	\E_{0,\mu}(J) := \rL^r_{\mu}(J;X_0).
\end{equation*}
We obtain the following embedding
\begin{equation}
	\E_{1,\mu}(J) \hookrightarrow \mathrm{BUC}(J;X_{\gamma,\mu}) . \label{eq:embedding mr}
\end{equation}
Next, we consider a linear operator $A \colon X_1 \to X_0$, a  bilinear operator $F \colon X_\beta \times X_\beta \to X_0$ and the abstract semilinear problem 
\begin{equation}
	\begin{cases}
		u' + Au = F(u,u)+f, \text{ for } t \in (0,T), \\
		u(0) = u_0 .
	\end{cases}
	\label{eq:acp-abstract}
	\tag{SLP}
\end{equation}
The following assumptions (A) guarantee existence and uniqueness of solutions to \eqref{eq:acp-abstract}. \\
{\bf Assumptions (A)}: 
\begin{enumerate}
\item[(MR)] $A$ has maximal $\rL^r$-regularity on $X_0$ for $r \in (1,+\infty)$.
\item[(H1)] $F \colon X_\beta \times X_\beta \to X_0$ satisfies the estimate
\begin{equation*}
						\| F(u_1,u_1) - F(u_2,u_2) \|_{X_0}
						\leq C \cdot ( \| u_1 \|_{X_\beta} +\| u_2 \|_{X_\beta}  ) \cdot \| u_1 - u_2 \|_{X_\beta} ,\quad \forall\ u_1, u_2 \in X_\beta.
\end{equation*}
\item[(H2)] $2 \beta + \frac{1}{r} \leq 1+\mu$. %$\beta - (\mu - \nicefrac{1}{p}) \leq \frac{1}{2} \cdot (1-(\mu-\nicefrac{1}{p}))$.
\item[(S)] $X_0$ is of class UMD, and the embedding 
		\begin{equation*}
			\rH^{1,q}(\R;X_0) \cap \rL^q(\R;X_1) \hookrightarrow \rH^{1-\beta,q}(\R;X_\beta)
		\end{equation*}
		holds for each $\beta \in (0,1)$ and $q \in (1,\infty)$.
	\end{enumerate}
In (H2) equality holds for $\mu_c = 2 \beta - 1 + \frac{1}{r}$. This weight $\mu_c$ is called the \emph{critical weight}.

\begin{proposition}\label{thm:local}
Let $u_0 \in X_{\gamma,\mu}$, $f \in \E_{0,\mu}(0,T)$. Suppose Assumption (A) and let $\mu \in (\mu_c,1]$. \\
	a) Then there exists $T' = T'(u_0) \in (0,T)$ such that \eqref{eq:acp-abstract} admits a unique solution $u \in \E_{1,\mu}(0,T')$. Moreover, the solution depends continuously on the data. \\
	b) If the spectral bound of $A$ is negative, i.e. if $s(A) < 0$, then there exists $r_0 > 0$ such that for all $\| u_0 \|_{X_{\gamma,\mu}} < r_0$, the solution $u$ from \cref{thm:local}(a) exists 
	globally, i.e. $T' = + \infty$. 
\end{proposition}

The embedding  \eqref{eq:embedding mr} yields the following blow-up criterion. Let $u$ be the solution from \cref{thm:local} and let $(0,t_+)$ be its maximal interval of existence. 
If $t_+ < T$, then 
\begin{equation*}
	\limsup_{t \uparrow t_+} \| u(t) \|_{X_{\gamma,\mu_c}} = + \infty. 
\end{equation*}
Furthermore, let $u$ be the solution from \cref{thm:local} and $(0,t_+)$ its maximal interval of existence. If $t_+ < T$, then
\begin{equation}
	\int_0^{t_+} \| u(t) \|_{X_{\mu_c}}^r \mathrm{d} t  = + \infty.
	\label{eq:serrin} 
\end{equation}
For a proof of the above results, we refer e.g. to \textcolor{black}{the work of Pr\"{u}ss, Wilke and Simonett \cite{PSW18}.} 

Finally, we note  the following version of Angenent's parameter trick \cite{Ang:90}. Note that in our situation the  non-linearity $F$ is bilinear and hence smooth. 

\begin{proposition}\label{thm:parameter trick}
Let $u$ be the solution from \cref{thm:local}.
	%	Assume $k \in \N$ and $q \in (1,\infty)$. If
	%	\begin{equation*}
		%		\E_{1,\mu} \to \E_{0,\mu} \colon u \mapsto F(u,u)
		%	\end{equation*}
	%	is $k$-times continuously differentiable. 
Assume that the linear problem
	\begin{equation*}
		%\begin{cases} 
			v' + Av + F(u,v) + F(v,u) = g,  v(0) = 0
		%\end{cases}  
	\end{equation*} 
admits for every $g \in \E_{0,\mu}(0,T')$, $0 < T' < T$ a unique solution $v \in \E_{1,\mu}(0,T')$.  If $g$ is of class $\rC^\infty$ or $\rC^\omega$, then $u \in \rC^\infty((0,T);X_1)$ or 
$u \in \rC^\omega((0,T);X_1)$.
\end{proposition}

\subsection{Function spaces}
\

Let $\Omega \subset \R^3$ be a domain with boundary $\partial \Omega$ and $\Gamma \subset \partial \Omega$. 
For $s > 0$ and $s \not \in \N$ we denote by $\rH^{s,q}(\Omega)$ the Bessel potential spaces, by $\rW^{s,q}(\Omega)$ the fractional Sobolev spaces and by $\rB_{q,r}^s(\Omega)$ the Besov spaces 
for $r \in [1,\infty]$. For the definition of these spaces we refer e.g. to \cite{Tri:78} or \cite{Ama:19}. The analogous spaces with respect to homogeneous Dirichlet boundary conditions at $\Gamma$ are 
defined as  

\begin{equation*}
	\rH_{\Gamma}^{s,q}(\Omega) 
	:=
	\begin{cases}
		\{ v \in \rH^{s,q}(\Omega) \colon v|_{\Gamma} = 0 \}, &\text{ if } s > \frac{1}{q}, \\ 
		\{ v \in \rH^{1/q,q}(\R^3) \colon \mathrm{supp}(v) \subset \Omega \cup \Gamma \}, &\text{ if } s = \frac{1}{q}, \\
		\qquad \rH^{s,q}(\Omega) &\text{ if } s < \frac{1}{q}.
	\end{cases} 
\end{equation*}
and
\begin{equation*}
	\rB_{q,r,\Gamma}^{s}(\Omega) 
	:=
	\begin{cases}
		\{ v \in \rB_{q,r}^{s}(\Omega) \colon v|_{\Gamma} = 0 \}, &\text{ if } s > \frac{1}{q}, \\
		\{ v \in \rB_{q,r}^{1/q}(\R^3) \colon \mathrm{supp}(v) \subset \Omega \cup \Gamma \}, &\text{ if } s = \frac{1}{q}, \\
		\qquad \rB_{q,r}^{s}(\Omega) &\text{ if } s < \frac{1}{q}. \\
	\end{cases} 
\end{equation*}
Note, that
\begin{equation*}
	\rB^{\frac{1}{q}}_{q,q,\Gamma}(\Omega) = \rW^{\frac{1}{q},q}_{00,\Gamma}(\Omega),
\end{equation*}
denotes the $\rL^q$-type {\em Lions-Magenes space} on $\Omega$ with respect to $\Gamma \subset \partial \Omega$ given by all $v \in \rW^{\frac{1}{q},q}(\Omega)$ such that 
\begin{equation*}
	\int_{\Omega} \rho(x)^{-1} | v(x) |^q \, \mathrm{d}x < \infty. 
\end{equation*}
Here $\rho$ denotes a smooth function comparable to the distance $\mathrm{dist}(\cdot,\Gamma)$ from $\Gamma$, equipped with the norm
\begin{equation*}
	\| v \|_{\rW_{00,\Gamma}^{\frac{1}{q},q}}
	= \biggl(\| v \|_{\rW^{\frac{1}{q},q}}^q  
	+ \int_{\Omega} \frac{|v(x)|^q}{d(x,\Gamma)^2} \, \mathrm{d}x \biggr)^{\frac{1}{q}} .
\end{equation*}
For more information on these spaces we refer e.g. to \cite{See:72} and \cite{MR4500197}.
In particular, for $q = 2$ we obtain the usual Lions-Magenes space  $\rH_{00,\Gamma}^{\frac{1}{2}}(\Omega)$ on $\Omega$ with respect to $\Gamma$, see \cite [Theorem 11.7, Page 66]{LM:72}. 

Finally, motivated by the critical spaces of the Navier-Stokes equations we consider the following critical Besov spaces with boundary conditions
\begin{equation*}
	\rB_{q,r,\Gamma}^{\frac{3}{q}-1}(\Omega) 
	:=
	\begin{cases}
		\{ v \in \rB_{q,r}^{\frac{3}{q}-1}(\Omega) \colon v|_{\Gamma} = 0 \}, &\text{ if } q > 2, \\
		\qquad \rH^{\frac{1}{2}}_{00,\Gamma}(\Omega), &\text{ if } q = 2, \\
		\qquad \rB_{q,r}^{\frac{3}{q}-1}(\Omega) &\text{ if } q < 2 . \\
	\end{cases} 
\end{equation*}

\subsection{Dirichlet operators and perturbation theory for boundary conditions}
\
\label{sec:boundary perturbation}

\textcolor{black}{In this section we make use of perturbation theory for boundary conditions which goes back in the abstract setting to the work of Greiner \cite{Gre:87}. To this end,    
let $X_0$ and $\partial X$ be Banach spaces and $\mathbb{A}_m \colon D(\mathbb{A}_m) \subset X_0 \to X_0$ be a linear %\emph{maximal} 
operator. Moreover, let $Z$ be a Banach space satisfying  $D(\mathbb{A}_m) \subset Z \subset X_0$,
%let $\partial X$ denote the {space of boundary conditions}, 
$\rL \colon D(\mathbb{A}_m) \to \partial X$ be a linear operator and $\Phi \in \mathcal{L}(Z,\partial X)$ be a bounded linear operator.
We call $X$ the \emph{state space}, $\partial X$ the \emph{boundary space}, $\mathbb{A}_m$ the \emph{maximal} operator, $L$ the \emph{boundary operator} and $\Phi$ the \emph{boundary perturbation operator}.}
%\emph{boundary operator}. Furthermore, let $Z$ be a Banach space such that $D(A_m) \subset Z \subset X_0$ and 
%$\Phi \in \mathcal{L}(Z,\partial X)$ the \emph{boundary perturbation operator}. 
The operator $\mathbb{A}_0 \colon X_1 \subset X_0 \to X_0$ {with homogeneous boundary conditions} is defined as 
\begin{equation*}
	\mathbb{A}_0 v = \mathbb{A}_m v, \enspace X_1 := D(\mathbb{A}_0) := \ker(\rL) .
\end{equation*}
Finally, we consider the operator $\mathbb{A} \colon D(\mathbb{A}) \subset X_0 \to X_0$ given by
\begin{equation*}
	\mathbb{A} v := \mathbb{A}_m v, \enspace D(\mathbb{A}) := \{ v \in D(\mathbb{A}_m) \colon \rL v = \Phi v \} . 
\end{equation*}

\medskip 

\textcolor{black}{Before we continue we briefly show how this abstract approach  can be applied to our setting. For further  details we refer to  \autoref{sec:linear}.
We choose  the state space $X_0 := \rL^q(\Omega_p) \times \rL^q_{\sigma}(\Omega_f)$ and the boundary space $\partial X = \rW^{1-\frac{1}{q},q}(\Gamma) \times \rW^{1-\frac{1}{q},q}(\Gamma)$.
Furthermore, we consider the maximal operator
\begin{equation*}
	\mathbb{A}_m \tbinom{p}{u} = 
	\tbinom{k \Delta p}{\mathcal{P} \Delta u}, \qquad 
	D(\mathbb{A}_m) = \rH^{2,q}_{\Gamma_1}(\Omega_p) \times (\rH^{2,q}_{\Gamma_2}(\Omega_f) \cap \rL^q_{\sigma}(\Omega_f))
\end{equation*} 
and choose $Z = \rH^{1,q}_{\Gamma_1}(\Omega_p) \times (\rH^{1,q}_{\Gamma_2}(\Omega_f) \cap \rL^q_{\sigma}(\Omega_f))$. Finally, the boundary operator $L$ and the boundary perturbation operator $\Phi$ are defined
by
\begin{equation*}
	\rL %\colon  \rH^{2,q}_{\Gamma_1}(\Omega_p) \times (\rH^{2,q}_{\Gamma_2}(\Omega_f) \cap \rL^q_{\sigma}(\Omega_f))
%	\to  \rW^{1-\frac{1}{q},q}(\Gamma) \times \rW^{1-\frac{1}{q},q}(\Gamma)
 \tbinom{p}{u}
	= \tbinom{\partial_{n_{\Gamma}} p}{\sigma n_{\Gamma}} 
	\quad \text{ and } \quad
	\Phi 
	%\colon %\rH^{1,q}_{\Gamma_1}(\Omega_p) \times (\rH^{1,q}_{\Gamma_2}(\Omega_f) \cap \rL^q_{\sigma}(\Omega_f)) \to  \rW^{1-\frac{1}{q},q}(\Gamma) \times \rW^{1-\frac{1}{q},q}(\Gamma)
	%	\colon 
	\tbinom{p}{u}	= \tbinom{- \frac{1}{k} u n_{\Gamma}}{-p n_{\Gamma} - \beta\sum_{\ell=1}^{2}(u\cdot t^{\ell}_{\Gamma})t^{\ell}_{\Gamma}}.
\end{equation*}
respectively. Let us point out that $\rL$ represents the leading order terms of the boundary conditions \eqref{bdry:BJ}, whereas $\Phi$ collects the lower order terms. The operator $\mathbb{A}_0$
with homogeneous boundary conditions is given by
\begin{equation*}
	\mathbb{A}_0 \tbinom{p}{u} := 
	\tbinom{k \Delta p}{\mathcal{P} \Delta u}, \qquad 
	D(\mathbb{A}_0) := \left\{ (p,u)^\top \in \rH^{2,q}_{\Gamma_1}(\Omega_p) \times (\rH^{2,q}_{\Gamma_2}(\Omega_f) \cap \rL^q_{\sigma}(\Omega_f)) \colon 
	\partial_{n_{\Gamma}} p = 0, \,\sigma n_{\Gamma} =0 
	\right\}
\end{equation*}
and finally, the operator $\mathbb{A}$ is given by
\begin{equation*}
	\mathbb{A} \binom{p}{u} := 
	\binom{k \Delta p}{\mathcal{P} \Delta u}, \qquad 
	D(\mathbb{A}) := \left\{ 
	\begin{aligned} 
	(p,u)^\top &\in \rH^{2,q}_{\Gamma_1}(\Omega_p) \times (\rH^{2,q}_{\Gamma_2}(\Omega_f) \cap \rL^q_{\sigma}(\Omega_f)) \\  
	&k\partial_{n_{\Gamma}} p = - un_{\Gamma}, \ \sigma n_{\Gamma} =-pn_{\Gamma} - \beta \sum_{\ell = 1}^2 (u \cdot t_{\Gamma}^{\ell}) t_{\Gamma}^{\ell}
	\end{aligned} 
	\right\} .
\end{equation*}
}

In the sequel, we assume that $\mathbb{A}_0$ generates a $C_0$-semigroup $(T_0(t))_{t \geq 0}$ on $X_0$.
Following Amann~\cite{Ama:95}, we introduce  an extrapolation scale and define the extrapolation space $X_{-1}$ as the completion of $X_0$ with respect to the norm induced by the 
resolvent, i.e. $X_{-1} := (X_0, \| R(\lambda,\mathbb{A}_0) \cdot \|_{X_0})^\sim$ for $\lambda \in \rho(\mathbb{A}_0)$.
Furthermore,  we define for $\beta \in (-1,1)$ the complex interpolation spaces $X_\beta$ as {\color{black}$X_\beta := [X_{-1},X_1]_{\frac{\beta+1}{2}}$.} The reiteration theorem guarantees that they  coincide with the 
definition given above for $\beta \in (0,1)$.

Furthermore, we denote by $F_\beta$ the Favard class associated with $\mathbb{A}_0$, see \cite[Section III.~5.b.]{EN00}. The Favard class can be characterized in terms of real interpolation spaces by 
$F_\beta = (X_0,X_1)_{\beta,\infty}$ for $\beta \in (0,1)$.
Moreover, we define the extrapolated semigroup $(T_{-1}(t))_{t \geq 0}$ on $X_{-1}$ via continuous extension from $(T_0(t))_{t \geq 0}$ on $X_0$. Its generator is the extrapolated operator $\mathbb{A}_{-1} \colon X_0 \subset X_{-1} \to X_{-1}$. Finally, we define the semigroups $(T_{\beta}(t))_{t \geq 0}$ on $X_{\beta}$ via restriction from the semigroup $(T_{-1}(t))_{t \geq 0}$ for $\beta \in (-1,1)$. Their generators are \textcolor{black}{the $X_\beta$-realizations of $\mathbb{A}_{-1}$, i.e. $\mathbb{A}_\beta u = \mathbb{A}_{-1}u$ with $D(\mathbb{A}_\beta) = \{ u \in X_0 \colon \mathbb{A}_{-1}u \in X_\beta \}$.}

\textcolor{black}{Following Greiner~\cite{Gre:87}} we define the \emph{Dirichlet operator} $\rL_\lambda := (\rL|_{\ker(\lambda-A_m)})^{-1}$, i.e.,
\begin{equation*}
	u = \rL_\lambda \varphi \enspace \Longleftrightarrow \enspace 
	\begin{cases}
		\lambda u = \mathbb{A}_m u, \\
		L u = \varphi .
	\end{cases}
\end{equation*}
he operator $\mathbb{A}$ can then be represented as
\textcolor{black}{
\begin{align*}
	 \mathbb{A} u &=	
	 (\mathbb{A}_{-1} + (\lambda-\mathbb{A}_{-1}) L_\lambda \Phi) u
	 =
	 (\mathbb{A}_{\delta-1} + (\lambda-\mathbb{A}_{-1}) L_\lambda \Phi) u ,\\
	D(\mathbb{A}) &= \{ u \in X_0 \colon   (\mathbb{A}_{-1} + (\lambda-\mathbb{A}_{-1}) L_\lambda \Phi)u \in X_0 \} 
	= \{ u \in X_{\delta} \colon   (\mathbb{A}_{\delta-1} + (\lambda-\mathbb{A}_{-1}) L_\lambda \Phi) u \in X_0 \} ,
\end{align*}
}
where $\delta \in [0,1)$ such that $D(\mathbb{A}_m) \subset X_\delta$, \textcolor{black}{see \cite[Lemma 3.4]{ABE:16} or \cref{lem:representation A}. For later use we recall the following result from the work of Adler, Bombieri and Engel {cf.~\cite[Corollary 3.7]{ABE:16}}.}

\begin{proposition}\label{prop:sw}
	Let $\mathbb{A}_0$ be generator of an analytic semigroup on $X_0$ of  angle $\theta \in (0,\nicefrac{\pi}{2}]$. Assume that there exists $0 \leq \gamma < \beta \leq 1$ such that 
	$D(\mathbb{A}_m) \subset F_{\beta}$ and % $X_\gamma \hookrightarrow Z$. 
	$[D(A_0-\omega)^\gamma]\hookrightarrow Z$. 
	Then $\mathbb{A}$ generates an analytic semigroup with angle $\theta \in (0,\nicefrac{\pi}{2}]$ on $X_0$.
\end{proposition}
Perturbations of his type are called Staffans-Wei{\ss} perturbations, see \cite{ABE:16, Wei:94, Sta:05}. 
Choosing $Z = X_\alpha$ we consider finally boundary perturbation of $\mathcal{R}$-sectorial operators.

\subsection{$\mathcal{R}$-sectorial operators}
\

The characterization of maximal $L^p$-regularity in terms of $\mathcal{R}$-boundedness of the associated resolvent allows to prove perturbation results for maximal $L^p$-regularity in UMD-spaces. 
For more information on $\mathcal{R}$-sectorial we refer to \cite{DHP:03}, \cite{KW:04} and \cite{PS16}.    
We are in particular interested in perturbations that are bounded in a fractional scale. 
More precisely, one can allow perturbations which are large in norm provided they are bounded from $X_\alpha$ to $X_{\beta-1}$ for some $0\leq \alpha < \beta \leq 1$. The case 
$\beta=1$ and $\alpha<1$ extends a well known perturbation theorem for generators of analytic semigroups to $\mathcal{R}$-sectorial operators and hence to maximal $L^p$-regularity in UMD-spaces.

\begin{proposition}\label{thm:boundary perturbation mr}
Let $\mathbb{A}_0$ be a $\mathcal{R}$-sectorial operator on $X_0$ of  angle $\theta_r(\mathbb{A}_0) < \theta < \pi$. Assume that there exists $0 \leq \gamma < \beta \leq 1$ such that 
$D(\mathbb{A}_m) \subset X_\beta$ and $X_\gamma \hookrightarrow Z$. Then $\mathbb{A}+\omega$ is $\mathcal{R}$-sectorial on $X_0$ of angle $\tilde{\theta} < \theta$ provided  $\omega > s(\mathbb{A})$. 
\end{proposition}

For a proof of Proposition \ref{thm:boundary perturbation mr} we refer e.g. to \cite[Corollary 12]{KW:01}. 

Note that $X_{\gamma'} \hookrightarrow D(\mathbb{A}_0-\omega)^\gamma \hookrightarrow X_\gamma$ for $\gamma' < \gamma$. If $\mathbb{A}_0 \in \mathcal{H}^\infty(X_0)$, then $X_\gamma = D(\mathbb{A}_0-\omega)^\gamma$.
Note also that $X_\beta \hookrightarrow F_\beta$. Hence, the conditions in \cref{prop:sw} are less restrictive than the conditions in \cref{thm:boundary perturbation mr}.

\section{Main results for Beaver-Joseph-Saffman interface conditions}
\label{sec:BJS}

In this section we discuss the Navier-Stokes-Darcy model with the Beaver-Joseph-Saffman interface condition. {\color{black}Our main result shows strong well-posedness for initial data in the critical Besov spaces.} 
{\color{black}
\begin{theorem}[Strong well-posedness in critical spaces]\label{thm:main bjs critical} \mbox{} \\
Let $r \in (1,\infty)$, $q \in (1,3)$ such that $\frac{2}{3 r} + \frac{1}{q} \leq 1$ and let $(p_0,u_0) \in 
\rB_{q,r,\Gamma_1}^{\frac{3}{q}-1}(\Omega_p)\times \big(\rB_{q,r,\Gamma_2}^{\frac{3}{q}-1}(\Omega_f) \cap \rL^q_{\sigma}(\Omega_f)\big)$, $(f_p,f_f)\in \rL^r_{\mu}(0,T;L^q(\Omega_p))\times \rL^r_{\mu}(0,T;L^q_{\sigma}(\Omega_f))$. Set $\mu = \frac{3}{2q}-\frac{1}{2}+ \frac{1}{r}$.
	\begin{enumerate}[(i)]
		\item 
		Then there exists $T>0$ such that the system \eqref{eq:ns-darcyp}--\eqref{eq:ns-darcyint} admits a unique, strong solution 
		\begin{align*}
			&p \in \rL^r_{\mu}(0,T; \rH^{2,q}(\Omega_p)) \cap \rH^{1,r}_{\mu}(0,T; \rL^q(\Omega_p))
			\cap \rC([0,T), \rB^{\frac{3}{q}-1}_{q,r}(\Omega_p)),\\
			&u \in 
			\rL^r_{\mu}(0,T; \rH^{2,q}(\Omega_f)) \cap \rH^{1,r}_{\mu}(0,T; \rL^q_{\sigma}(\Omega_f))
			\cap \rC([0,T), \rB^{\frac{3}{q}-1}_{q,r}(\Omega_p)). 
		\end{align*}
		%where $\mu = \frac{3}{2q}-\frac{1}{2}+ \frac{1}{r}$.
		\item  %Set $\mu = \frac{3}{2q}-\frac{1}{2}+ \frac{1}{r}$.
		There exists a small constant $r_0 > 0$ such that for initial data $(p_0,u_0)$ 
		with \\ 
$\| p_0 \|_{\rB_{q,r}^{\frac{3}{q}-1}}\hspace{-0.2em} + \| u_0 \|_{\rB_{q,r}^{\frac{3}{q}-1}}\hspace{-0.2em} < r_0$ the system \eqref{eq:ns-darcyp}--\eqref{eq:ns-darcyint} admits a 
unique,  global, strong solution 
		\begin{align*}
			&p \in \rL^r_{\mu}(\R_+; \rH^{2,q}(\Omega_p)) \cap \rH^{1,r}_{\mu}(\R_+; \rL^q(\Omega_p))
			\cap \rC([0,\infty), \rB^{\frac{3}{q}-1}_{q,r}(\Omega_p)),\\
			&u \in 
			\rL^r_{\mu}(\R_+; \rH^{2,q}(\Omega_f)) \cap \rH^{1,r}_{\mu}(0,T; \rL^q_{\sigma}(\Omega_f))
			\cap \rC([0,\infty), \rB^{\frac{3}{q}-1}_{q,r}(\Omega_f)), 
		\end{align*}
		%where $\mu = \frac{3}{2q}-\frac{1}{2}+ \frac{1}{r}$.
	\end{enumerate} 
\end{theorem}
}

\begin{remark}\rm
We point out that, thanks to the regularity of the strong solution in \cref{thm:main bjs critical}, the solution satisfies the interface conditions, in particular the Beavers-Joseph-Saffman conditions \eqref{eq:ns-darcyint}$_3$, in the sense of traces. 
\end{remark}

Blow ups in finite time are characterized by the following result.

\begin{cor}[Finite-in-time-Blow up]\label{cor:bjs blow up}
	Assume that $(p_0, u_0) \in %(\bm{X}_0,\mb{X}_1)_{\frac{1}{4},2} \subseteq  
	\rB_{r,q,\Gamma_1}^{\frac{3}{q}-1}(\Omega_p)\times (\rB_{r,q,\Gamma_2}^{\frac{3}{q}-1}(\Omega_f) \cap \rL^q_{\sigma}(\Omega_f))$, $(f_p,f_f)\in \rL^r_{\mu}(0,T;L^q(\Omega_p))\times \rL^r_{\mu}(0,T;L^q_{\sigma}(\Omega_f))$ and let $(p,u)$ be the unique local solution from \cref{thm:main bjs critical} on the maximal time interval $[0,t_+)$.
	\begin{itemize}
		\item[(i)] 
		If $(p,u)$ blows up in finite time, i.e. $t_+ < \infty$, then
		\begin{equation*}
			\limsup_{t \to t_+} 
			\biggl( \| p(t) \|_{\rB^{\frac{3}{q}-1}_{q,r,\Gamma_1}} \hspace{-0.5em} + \| u(t) \|_{\rB^{\frac{3}{q}-1}_{q,r,\Gamma_2}} \biggr) 
			= + \infty .
		\end{equation*}
		\item[(ii)] 
		If $(p,u)$ blows up in finite time, i.e. $t_+ < \infty$, then 
		\begin{equation*}
			\limsup_{T \to t_+} \int_0^{T} 
			\biggl(\| p(s) \|_{\rH^{2\mu,q}}^r + \| u(s) \|_{\rH^{2\mu,q}}^r \mathrm{d} s \biggr) = + \infty ,
		\end{equation*}
		where $\mu := \frac{3}{2q}+\frac{1}{r}-\frac{1}{2}$.
	\end{itemize} 
\end{cor}

Finally, the following results show that additional regularity of the external forces yields additional regularity of the solution. 
The first interior regularity result is not really surprising, since interior regularity is not affected by the interface conditions. 

\begin{cor}[Higher interior regularity]\label{thm:reg in bjs}
	Assume $r \in (1,\infty)$ and $q \in (5/2,3)$ such that $\frac{2}{r} + \frac{1}{q} \leq 1$.
	Let $(p,u)$ be the unique solution to the initial data $(p_0,u_0) \in	\rB_{r,q,\Gamma_1}^{\frac{3}{q}-1}(\Omega_p)\times (\rB_{r,q,\Gamma_2}^{\frac{3}{q}-1}(\Omega_f) \cap \rL^q_{\sigma}(\Omega_f))$ from \cref{thm:main bjs critical}.
	\begin{enumerate}[(a)]
		\item If $(f_p,f_f) \in \rC^k((0,T),\rL^q(\Omega_p) \times \rL^q_{\sigma}(\Omega_f))$ for $k \in \{\infty,\omega\}$, where $\rC^\omega$ denotes the real analytic functions, then
		%or  $(f_p,f_f) \in \rC^{\omega}((0,T),\rL^q(\Omega_p) \times \rL^q_{\sigma}(\Omega_f))$, then 
		\begin{align*}
			p \in \rC^k((0,T); \rH^{2,q}(\Omega_p)) \quad &\text{ and } \quad u \in 
			\rC^k((0,T); \rH^{2,q}(\Omega_f)\cap \rL^q_{\sigma}(\Omega_f)) .
			%				, \\
			%				\text{ or } \quad p \in \rC^\omega((0,T); \rH^{2,q}(\Omega_p)) \quad &\text{ and } \quad u \in 
			%				\rC^\omega((0,T); \rH^{2,q}(\Omega_f)\cap \rL^q(\Omega_f)) .
		\end{align*}
		\item If $f_p \in \rC^k((0,T) \times\Omega_p)$ and $f_f \in \rC^k((0,T) \times \Omega_f)$ for $k \in \{\infty,\omega\}$, then
		% or $f_p \in \rC^\omega((0,T) \times\Omega_p)$ and $f_f \in \rC^\omega((0,T) \times \Omega_f)$, then 
		\begin{align*}
			p \in \rC^k((0,T) \times \Omega_p) \quad &\text{ and } \quad u \in 
			\rC^k((0,T) \times \Omega_f) .
			%, \\
			%	\text{ or } \quad	p \in \rC^\omega((0,T) \times\Omega_p) \quad &\text{ and } \quad u \in 
			%	\rC^\omega((0,T)\times \Omega_f) .
		\end{align*}
	\end{enumerate}
\end{cor}

The next result concerns  the regularity at the interface. Assuming  that $\Gamma$ is analytic, we  obtain the following result.  

\begin{cor}[Higher regularity at the interface]\label{thm:reg bjs}
Assume that $\Gamma$ is analytic and let  $\frac{1}{r}+\frac{3}{2q} < 1$.
	Let $(p,u)$ be the unique solution for  $(p_0,u_0) \in	\rB_{r,q,\Gamma_1}^{\frac{3}{q}-1}(\Omega_p)\times (\rB_{r,q,\Gamma_2}^{\frac{3}{q}-1}(\Omega_f) \cap \rL^q_{\sigma}(\Omega_f))$ 
from \cref{thm:main bjs critical}.

	If $f_p \in \rC^k((0,T) \times\Gamma)$ and $f_f \in \rC^k((0,T) \times \Gamma)$ for $k \in \{\infty,\omega\}$, then
	%or $f_p \in \rC^\omega((0,T) \times\Gamma)$ and $f_f \in \rC^\omega((0,T) \times \Gamma)$, then 
	\begin{align*}
		p \in \rC^k((0,T) \times \Gamma) \quad &\text{ and } \quad u \in 
		\rC^k((0,T) \times \Gamma) .
		%				, \\
		%				\text{ or } \quad	p \in \rC^\omega((0,T) \times\Gamma) \quad &\text{ and } \quad u \in 
		%				\rC^\omega((0,T)\times \Gamma) .
	\end{align*}
\end{cor}

\begin{remark}\rm
Let us point out that \cref{thm:reg bjs} is about regularity in tangential direction at the interface $\Gamma$. 
Furthermore, note that in \cref{thm:reg in bjs} and \cref{thm:reg bjs} the Hilbert space case is excluded which is a main reason to consider the 
	%This is one reason why we deal with the problem in the 
	$\rL^r_t$-$\rL^q_x$-setting here.
\end{remark}

\section{Main results for Beaver-Joseph interface conditions}
\label{sec:BJ}
	
	In this section we discuss the Navier-Stokes-Darcy model with Beaver-Joseph conditions. 
	Our main result shows strong well-posedness for initial data in the critical Lions-Magenes spaces. 
	It can be seen as an analogue for our setting 
	to the celebrated result of Fujita-Kato \cite{FK:64}.
	%results of Giga \cite{Gig:86}, Fujita-Kato \cite{FK} and Kato \cite{Kato:84} about Navier-Stokes equations.
	
	\begin{theorem}[Strong well-posedness in critical spaces]\label{cor:bj r=q=2} \mbox{}\\
		Assume that $\beta > 0$ is small enough. Let $(p_0,u_0) 
		\in 
		\rH^{\frac{1}{2}}_{00,\Gamma_1}(\Omega_p)\times \big( \rH^{\frac{1}{2}}_{00,\Gamma_2}(\Omega_f) \cap \rL^2_{\sigma}(\Omega_f) \big)$, $(f_p,f_f)\in \rL^r_{\mu}(0,T;L^q(\Omega_p))\times \rL^r_{\mu}(0,T;L^q_{\sigma}(\Omega_f))$. 
		\begin{enumerate}[(i)]
			\item 
			Then there exists $T>0$ such that the system \eqref{eq:ns-darcyp}--\eqref{eq:ns-darcyint}$_2$ with Beavers-Joseph condition \eqref{bdry:BJ}  admits a unique, strong solution 
			\begin{align*}
				&p \in \rL^2_{\frac{3}{4}}(0,T; \rH^2(\Omega_p)) \cap \rH^1_{\frac{3}{4}}(0,T; \rL^2(\Omega_p))
				\cap \rC([0,T), \rH^{\frac{1}{2}}_{00,\Gamma_1}(\Omega_p))
				\cap \rC((0,T), \rH^1_{0,\Gamma_1}(\Omega_p)) ,\\
				&u \in \rL^2_{\frac{3}{4}}(0,T;\rH^{2}(\Omega_f)) \cap \rH^{1}_{\frac{3}{4}}(0,T;\rL_{\sigma}^2(\Omega_f))\cap \rC([0,T),\rH_{00,\Gamma_2}^{\frac{1}{2}}(\Omega_f))
				\cap \rC((0,T),\rH^1_{0,\Gamma_2}(\Omega_f)) .
			\end{align*}
			\item 
			There exists a small constant $r_0 > 0$ such that for initial data $(p_0,u_0)$  
			with $\| p_0 \|_{\rH^{\frac{1}{2}}_{00,\Gamma_1}}\hspace{-1em} + \| u_0 \|_{\rH^{\frac{1}{2}}_{00,\Gamma_2}} \hspace{-1em} < r_0$ the system \eqref{eq:ns-darcyp}--\eqref{eq:ns-darcyint}$_2$ with Beavers-Joseph condition \eqref{bdry:BJ} admits a unique, global, strong solution 
			\begin{align*}
				&p \in \rL^2_{\frac{3}{4}}(\R_+; \rH^2(\Omega_p)) \cap \rH^1_{\frac{3}{4}}(\R_+; \rL^2(\Omega_p))
				\cap \rC([0,\infty), \rH^{\frac{1}{2}}_{00,\Gamma_1}(\Omega_p))
				\cap \rC((0,\infty), \rH^1_{0,\Gamma_1}(\Omega_p)) ,\\
				&u \in \rL^2_{\frac{3}{4}}(\R_+;\rH^{2}(\Omega_f)) \cap \rH^{1}_{\frac{3}{4}}(\R_+;\rL_{\sigma}^2(\Omega_f))\cap \rC([0,\infty),\rH_{00,\Gamma_2}^{\frac{1}{2}}(\Omega_f))
				\cap \rC((0,\infty),\rH^1_{0,\Gamma_1}(\Omega_f)) .
			\end{align*}
		\end{enumerate}
	\end{theorem}

\begin{remark}\rm
	We point out that, thanks to the regularity of the strong solution in \cref{cor:bj r=q=2}, the solution satisfies the interface conditions, in particular the Beavers-Joseph conditions \eqref{bdry:BJ}, in the sense of traces. 
\end{remark}
	
	For arbitrary local solutions of the Navier Stokes-Darcy model there are two possibilities: either they exists globally or the blow up in finite time. 
	Blow ups in finite time are characterized by the following result.

	\begin{cor}[Finite-in-time-Blow up]\label{cor:bj blow up}
Assume that $\beta > 0$ is sufficient small and $(p_0,u_0) \in \rH^{\frac{1}{2}}_{00,\Gamma_1}(\Omega_p)\times \big( \rH^{\frac{1}{2}}_{00,\Gamma_2}(\Omega_f) \cap \rL^2_{\sigma}(\Omega_f) \big)$, $(f_p,f_f)\in \rL^r_{\mu}(0,T;L^q(\Omega_p))\times \rL^r_{\mu}(0,T;L^q_{\sigma}(\Omega_f))$ and let 
$(p,u)$ be the unique, local solution from \cref{cor:bj r=q=2} on the maximal time interval $[0,t_+)$.
		\begin{itemize}
			\item[(i)] 
			If $(p,u)$ blows up in finite time, i.e. $t_+ < \infty$, then
			\begin{equation*}
				\limsup_{t \to t_+} 
				\biggl( \| p(t) \|_{\rH^{\frac{1}{2}}_{00,\Gamma_1}} \hspace{-1em} + \| u(t) \|_{\rH^{\frac{1}{2}}_{00,\Gamma_2}} \biggr) 
				= + \infty .
			\end{equation*}
			\item[(ii)] 
			If $(p,u)$ blows up in finite time, i.e. $t_+ < \infty$, then 
			\begin{equation*}
				\limsup_{T \to t_+} \int_0^{T} 
				\biggl(\| p(s) \|_{\rH^{\frac{3}{2}}}^2 + \| u(s) \|_{\rH^{\frac{3}{2}}}^2 \mathrm{d} s \biggr) = + \infty .
			\end{equation*}
			%where $\mu := \frac{3}{2q}+\frac{1}{r}-\frac{1}{2}$.
		\end{itemize} 
	\end{cor}

\section{The linearized system}
\label{sec:linear}

In this section, we analyze the linearized system subject to Beaver-Joseph-Saffman and Beaver-Joseph interface conditions. 
To this end, we make use of  extrapolation scales of Banach spaces, see e.g. the work of Amann \cite{Ama:95}, and perturbation theory of boundary operators via abstract Dirichlet operators.
We begin with the situation of pure Neumann boundary conditions.

\subsection{Laplacian and Stokes operator with Neumann boundary conditions} 
\label{ssec:Neumann}
\

Consider for $q \in (1,\infty)$ the Banach space 
\begin{equation*}
	{X}_0= \rL^q(\Omega_p)\times \rL^q_{\sigma}(\Omega_f),
\end{equation*}
and define the maximal Laplacian $\Delta_m:D(\Delta_m) \subset \rL^q(\Omega_p) \rightarrow \rL^q(\Omega_p)$ as
\begin{equation}\label{eq:Delta_m}
	\Delta_m p=\Delta p \quad \mbox{ with domain } \quad D(\Delta_m)= \rH^{2,q}_{\Gamma_1}(\Omega_p),
\end{equation}
and the maximal Stokes operator $A_m \colon D(A_m) \subset \rL^q_{\sigma}(\Omega_f) \to \rL^q_{\sigma}(\Omega_f)$ as 
\begin{equation}\label{eq:stokes}
	A_m u= \mathcal{P} \Delta u
	%\D\sigma_f,
	%\{u\in \mb{L}^2_{\sigma}(\Omega_f)\mid u=0\ \mbox{on}\ \Gamma_f^{ext}\},
	\quad \text{ with domain } \quad 
	D(A_m)=
	\rH^{2,q}_{\Gamma_2}(\Omega_f) \cap \rL^q_{\sigma}(\Omega_f) ,
	%\Big\{u\in \mb{L}^2_{\sigma}(\Omega_f)\cap \mb{H}^{2}(\Omega_f)\mid u=0\ \mbox{on}\ (0,T)\times\Gamma_f^{ext},\   \exists\ p_f\in H^{1}(\Omega_f)\mbox{ such that } \D\sigma_f\in \mb{L}^2(\Omega_f) \Big\}.%\mbox{ and }\sigma_fn_f=0\ \mbox{on}\ \Gamma_f^{in}\cup\Gamma_f^{out} \Big\}.
\end{equation}
where $\mathcal{P}\colon {L}^q(\Omega_f) \to {L}^q_{\sigma}(\Omega_f)$ denotes the Helmholtz projection.
\textcolor{black}{With regard to the theory of boundary perturbations discussed in \autoref{sec:boundary perturbation}, we define the maximal operator 
\begin{equation*}
	\mb{A}_m := 
	\begin{pmatrix} 
		k \Delta_m & 0 \\ 0 & A_m
	\end{pmatrix} 
	\text{ with diagonal domain } D(\mb{A}_m) := D(\Delta_m) \times D(A_m) .
\end{equation*}
}

We denote by $\Delta_0 \colon D(\Delta_0)  \subset \rL^q(\Omega_p) \to \rL^q(\Omega_p)$ the Laplacian with mixed boundary conditions, i.e.
\begin{equation*}
	\Delta_0 p = \Delta p  \mbox{ with domain } \quad D(\Delta_0)= \{ p \in \rH^{2,q}(\Omega_p) \mid \partial_{n_{\Gamma}} p = 0 \text{ and } p|_{\Gamma_1} = 0 \} ,
\end{equation*}
and by $A_0 \colon D(A_0) \subset \rL^q_{\sigma}(\Omega_f) \to \rL^q_{\sigma}(\Omega_f)$ the Stokes operator with mixed boundary conditions, i.e.
\begin{equation*}
	A_0 u = \mathcal{P} \Delta u
	\quad \text{ with domain } \quad 
	D(A_0)=
	\{ u \in \rH^{2,q}(\Omega_f) \cap \rL^q_{\sigma}(\Omega_f) \mid \sigma n_{\Gamma} = 0 \text{ and } u|_{\Gamma_2} = 0 \}.
\end{equation*} 
Here $\Gamma_1$ and $\Gamma_2$ are defined as in Section 2. 

Furthermore, we define the operator  $\mb{A}_0 \colon {Z}_1 \subset {X}_0 \to {X}_0$ by
\begin{equation*}
	\mb{A}_0 := 
	\begin{pmatrix}
		k \Delta_0 & 0 \\
		0 & A_0
	\end{pmatrix}
\end{equation*}
with diagonal domain
\begin{equation*}
	{Z}_1 := D(\Delta_0) \times D(A_0).
\end{equation*}

{\color{black}
\begin{remark}\label{pressure}
Note that    the pressure $\pi$ in \eqref{eq:ns-darcyu} can be represented as $\nabla \pi = (I-\mathcal{P})\Delta u$. We hence obtain  $\operatorname{div}\sigma= \mathcal{P} \Delta u= A_m u$.
\end{remark}
}
We now collect properties of $\mb{A}_0$ and the associated real and complex interpolation spaces.

\begin{lemma}\label{lem:A_0}
\begin{enumerate}[(a)]
\item The operator  $\mb{A}_0$ admits $\rL^r$-maximal regularity on ${X}_0$ for all $r \in (1,\infty)$.
\item The complex interpolation spaces
		${Z}_{\theta} = [{Z}_1,{X}_0]_{\theta}$ are given by
\begin{equation*}
			{Z}_{\theta} = \rH^{2\theta,q}_{\Gamma_1}(\Omega_p) \times \bigl( \rH^{2\theta,q}_{\Gamma_2}(\Omega_f) \cap \rL^q_{\sigma}(\Omega_f) \bigr) 
			\qquad \text{ for } 0 < \theta < \frac{1}{2}\Big(1+\frac{1}{q}\Big).
			%				\begin{cases}
				%					%
				%				%	\{p \in \rH^{2\theta,q}(\Omega_p) \mid \partial_{n_{\Gamma}} p = 0, \, p|_{\Gamma_1} = 0 \} \times (\{ u \in \rH^{2\theta,q}(\Omega_f) \mid \sigma_f \nu_{\Gamma} =0, \, u|_{\Gamma_2}  = 0 \}~\cap~\rL^q_{\sigma}(\Omega_f)), &\text{ if } \theta \in \left(\frac{1}{2} \cdot \left(1+\frac{1}{q}\right),1 \right) \\
				%				%	\textcolor{red}{???}, &\text{ if } \theta = \frac{1}{2} \cdot \left(1+\frac{1}{q} \right) \\
				%					%
				%					\{p \in \rH^{2\theta,q}(\Omega_p) \colon p|_{\Gamma_1} = 0 \} \times (\{ u \in \rH^{2\theta,q}(\Omega_f) \colon u|_{\Gamma_2}  = 0 \}~\cap~\rL^q_{\sigma}(\Omega_f)), &\text{ if } \theta \in \left(\frac{1}{2q},\frac{1}{2} \cdot \left(1+\frac{1}{q}\right) \right) \\
				%					%
				%					\rH_{00,\Gamma_1}^{\frac{1}{q},q}(\Omega_p) \times ( \rH_{00,\Gamma_2}^{\frac{1}{q},q}(\Omega_f) ~\cap~\rL^q_{\sigma}(\Omega_f)), &\text{ if } \theta = \frac{1}{2q}\\
				%					\rH^{2\theta,q}(\Omega_p) \times (\rH^{2\theta,q}(\Omega_f) ~\cap~\rL^q_{\sigma}(\Omega_f)), &\text{ if } \theta < \frac{1}{2q} .
				%				\end{cases}
		\end{equation*}
\item The real interpolation spaces ${Z}_{\theta,s} = ({Z}_1,{X}_0)_{\theta,s}$ 
		%for $\theta \in (0,1/2(1+1/q))$ and $r \in (1,\infty)$ 
		are given by
		\begin{equation*}
			{Z}_{\theta,r} = 
			\rB^{2\theta}_{qr\Gamma_1}(\Omega_p) \times \bigl( \rB^{2\theta}_{qr,\Gamma_2}(\Omega_f) \cap \rL^q_{\sigma}(\Omega_f) \bigr) 
			\qquad \text{ for } 0 < \theta < \frac{1}{2}\Big(1+\frac{1}{q}\Big)
			\text{ and } r \in (1,\infty) .			
			%				\begin{cases}
				%					\{p \in \rB_{qs}^{2\theta}(\Omega_p) \colon p|_{\Gamma_1} = 0 \} \times (\{ u \in \rB_{qs}^{2\theta}(\Omega_f) \colon u|_{\Gamma_2}  = 0 \}~\cap~\rL^q_{\sigma}(\Omega_f)), &\text{ if } \theta \in \left(\frac{1}{2q},\frac{1}{2} \cdot \left(1+\frac{1}{q}\right) \right) \\
				%					%
				%					\rB_{qs,00,\Gamma_1}^{\frac{1}{q}}(\Omega_p) \times ( \rB_{qs,00,,\Gamma_2}^{\frac{1}{q}}(\Omega_f) ~\cap~\rL^q_{\sigma}(\Omega_f)),
				%					\textcolor{red}{?} &\text{ if } \theta = \frac{1}{2q}\\
				%					\rB_{qs}^{2\theta}(\Omega_p) \times (\rB_{qs}^{2\theta}(\Omega_f) ~\cap~\rL^q_{\sigma}(\Omega_f)), &\text{ if } \theta < \frac{1}{2q} .
				%				\end{cases} 
		\end{equation*}
	\end{enumerate}
\end{lemma}	
\begin{proof}
(a) The maximal $\rL^r$-regularity of the Stokes operator $A_0$ with mixed boundary conditions is proved in \cite[Theorem 1.1]{Pru:18}. The analogous result for Laplacian can be found e.g. in \cite{DHP:03}. 

(b) The characterization of the complex interpolation spaces associated with the Stokes operator $A_0$ with mixed boundary conditions follows from \cite[Theorem 1.1]{Pru:18}, if $\theta \not = \frac{1}{2q}$. 
For the critical Dirichlet exponent $\theta = \frac{1}{2q}$ it follows from the proof of \cite[Theorem 1.1]{Pru:18} that it coincides with the critical interpolation space only taking into account the Dirichlet boundary conditions at $\Gamma_2$. In the work of Giga \cite{Gig:85} it is shown that this space is the intersection between the critical complex interpolation space for the Dirichlet Laplacian and the solenoidal vector-fields $\rL^q_{\sigma}(\Omega_f)$. The critical interpolation space for the Dirichlet Laplacian is characterized in \cite{See:72}. 
		For the complex interpolation spaces associated with the Laplacian we see again that it is reduced to the complex interpolation spaces of the Laplacian only taking into account the Dirichlet boundary condition at $\Gamma_1$, since we only consider $\theta < \frac{1}{2} \cdot \bigl( 1 + \frac{1}{q} \bigr)$. These spaces can be found in \cite{See:72}.  

(c) As in (b) we only have to take into account the Dirichlet boundary conditions on $\Gamma_1$ and $\Gamma_2$, respectively, since $\theta < \frac{1}{2} \cdot \bigl( 1 + \frac{1}{q} \bigr)$. Then 
the characterization of these spaces can be found in \cite{Ama:00}. 
%\textcolor{red}{In general Amman's article from 2000 in JMFM seems to characterize all the spaces we need here. 
%He also characterizes the critical space $\theta = \frac{1}{2q}$, see (2.19) and rmk 3.7(b). Combining this with our knowledge from the proof of PrÃŒss that we can do both boundary conditions 
%separately we obtain the claim (probably this can be also found in the works of Amann)... Unfortunately, I cannot see why his critical space coincides with Lions-Magenes... }
\qedhere 
	%Feffermann \cite{MR4500197}.
\end{proof}

Next, we consider the \textcolor{black}{harmonic extension problem with inhomogenenous Neumann data}
%eigenvalue problem for $\lambda=1$ for the Laplacian with inhomogeneous Neumann data
\begin{equation}
	\left\{ 
	\begin{aligned}
		&k \Delta_m p = \textcolor{black}{0}, \\
		\partial_{n_{\Gamma}} p = &\varphi, \qquad 
		p|_{\Gamma_1} = 0 
	\end{aligned}
	\right.
	\label{eq:laplacian neumann problem}
\end{equation}	
for $\varphi \in \mathrm{W}^{1-\frac{1}{q},q}(\Gamma)$, as well as the \textcolor{black}{stationary Stokes problem }
%eigenvalue problem for $\lambda=1$ for the Stokes operator 
with inhomogeneous Neumann data 
\begin{equation}
	\left\{ 
	\begin{aligned}
		&A_m u = \textcolor{black}{0}, \\
		\sigma n_{\Gamma} = &\psi, \qquad 
		u|_{\Gamma_2} = 0 
	\end{aligned}
	\right.
	\label{eq:stokes neumann problem}
\end{equation}	
for $\psi \in \mathrm{W}^{1-\frac{1}{q},q}(\Gamma)$. We then  obtain the following assertions on the solvablity of problems \eqref{eq:laplacian neumann problem} and 
\eqref{eq:stokes neumann problem}.

{\color{black}The problem \eqref{eq:laplacian neumann problem} admits a unique solution $p\in \rH^{2,q}(\Omega_p)$ for all $\varphi \in \rH^{1-\frac{1}{q}}(\Gamma)$ by \cite[Corollary 7.4.5]{PS16}}. We denote 
the solution operator 
\begin{equation*} 
\varphi \mapsto p \mbox{ by } \rL_0^\Delta \varphi := p
\end{equation*} 
and note that $\rL_0^\Delta \in \mathcal{L}(\mathrm{W}^{1-\frac{1}{q},q}(\Gamma),\rH^{2,q}(\Omega_p))$.

By similar arguments, the problem \eqref{eq:laplacian neumann problem} admits a unique solution $u \in \rH^{2,q}(\Omega_f)\cap \rL^q_{\sigma}(\Omega_f)$ for 
all $\psi \in \rH^{1-\frac{1}{q}}(\Gamma)$. We denote the associated solution operator 
\begin{equation*} 
\psi \mapsto u \mbox{ by } \rL_0^A \psi := u 
\end{equation*} 
and note that $\rL_0^A \in \mathcal{L}(\mathrm{W}^{1-\frac{1}{q},q}(\Gamma),\rH^{2,q}(\Omega_f)\cap \rL^q_{\sigma}(\Omega_f))$.
 
The operator 
\begin{equation*} 
\rL_0 (\varphi,\psi)^{\top} := (\rL_0^\Delta \varphi,\rL_0^A \psi)^{\top}
\end{equation*}  
is bounded from $\mathrm{W}^{1-\frac{1}{q}}(\Gamma) \times \mathrm{W}^{1-\frac{1}{q}}(\Gamma)$ to $\rH^{2,q}(\Omega_p) \times (\rH^{2,q}(\Omega_f)~\cap~\rL^q_{\sigma}(\Omega_f))$. This follows by combining 
the two assertions above. 
%\textcolor{red}{Amann, Pruess...} Page 360 PrÃŒss book

%	 We consider the following Banach space % $\mb{X}_0$:
%		\begin{equation*}
	%			\mb{X}_0= \rL^q(\Omega_p)\times {L}^q_{\sigma}(\Omega_f),
	%		\end{equation*}
%		for $q \in (1,\infty)$.
%		Next, we define the following operators, the Laplacian $\Delta_m:D(\Delta_m) \subset \rL^q(\Omega_p) \rightarrow \rL^q(\Omega_p)$ as
%		\begin{equation}\label{eq:Delta_m}
	%			\Delta_m q=\Delta q \quad \mbox{ with domain } \quad D(\Delta_m)= \rH^{2,q}(\Omega_p),
	%		\end{equation}
%		and the Stokes operator $A_m \colon D(A_m) \subset \rL^q_{\sigma}(\Omega_f) \to \rL^q_{\sigma}(\Omega_f)$ as 
%		\begin{equation}\label{eq:stokes}
	%			A_m u= \mathcal{P} \Delta u
	%			%\D\sigma_f,
	%			%\{u\in \mb{L}^2_{\sigma}(\Omega_f)\mid u=0\ \mbox{on}\ \Gamma_f^{ext}\},
	%			\quad \text{ with domain } \quad 
	%			D(A_m)=
	%			\rH^{2,q}(\Omega_f) \cap \rL^q_{\sigma}(\Omega_f) ,
	%			%\Big\{u\in \mb{L}^2_{\sigma}(\Omega_f)\cap \mb{H}^{2}(\Omega_f)\mid u=0\ \mbox{on}\ (0,T)\times\Gamma_f^{ext},\   \exists\ p_f\in H^{1}(\Omega_f)\mbox{ such that } \D\sigma_f\in \mb{L}^2(\Omega_f) \Big\}.%\mbox{ and }\sigma_fn_f=0\ \mbox{on}\ \Gamma_f^{in}\cup\Gamma_f^{out} \Big\}.
	%		\end{equation}
%		where $\mathcal{P}\colon {L}^q(\Omega_f) \to {L}^q_{\sigma}(\Omega_f)$ denotes the Helmholtz projection.

\subsection{The linearized system subject to Beaver-Joseph-Saffman boundary conditions}
\

The left-hand side of the conditions \eqref{eq:ns-darcyint}$_2$ and \eqref{eq:ns-darcyint}$_3$ describe the normal part and the tangential part of $\sigma n_\Gamma$. They are equivalent to 
\begin{equation*}
	\sigma n_{\Gamma} = -pn_{\Gamma} - \beta\sum_{\ell=1}^{2}(u\cdot t^{\ell}_{\Gamma})t^{\ell}_{\Gamma} .
\end{equation*}
This motivates the definition of the operator $\mb{A}:{X}_1 \subset {X}_0\rightarrow {X}_0$ by
\begin{equation}\label{def:opmatrix}
	\mb{A}=\begin{pmatrix}
		k \Delta_m & 0 \\  0 & A_m
	\end{pmatrix}
\end{equation}
with non-diagonal domain
\begin{equation}\label{dom:opmatrix}
	{X}_1=\left\{\binom{p}{u}\in D(\Delta_m)\times D(A_m)\left|
	\begin{matrix} 
		\sigma n_{\Gamma} = -pn_{\Gamma} - \beta\sum_{\ell=1}^{2}(u\cdot t^{\ell}_{\Gamma})t^{\ell}_{\Gamma}, \\
		u\cdot n_{\Gamma}=-k\nabla p\cdot n_{\Gamma}, \qquad  %&\mbox{on}\ \Gamma  \\
		%&\mbox{on}\ \Gamma \\
		p|_{\Gamma_1} = 0, \qquad %&\mbox{on}\ \Gamma_1 \\
		u|_{\Gamma_2} = 0 %&\mbox{on}\ \Gamma_2
	\end{matrix} \right. \right\} .
\end{equation}
We then obtain the following properties of $\mb{A}$ and its related interpolation spaces.

\begin{proposition}\label{prop:mr}
\begin{enumerate}[(a)]
\item 
The operator  $\mb{A}$ has compact resolvent on ${X}_0$; its spectrum is $q$-independent and consists only of eigenvalues $\lambda$ satisfying  $\mathrm{Re}(\lambda) < 0$. In particular, 
$\mb{A}$ is boundedly invertible. 	
\item 
The operator  $\mb{A}$ admits $\rL^r$-maximal regularity on ${X}_0$ for all $r \in (1,\infty)$.
\item 
For  $0 < \theta < \nicefrac{1}{2} \cdot \bigl(1+\nicefrac{1}{q}\bigr)$ and  $r \in (1,\infty)$ the real interpolation spaces ${X}_{\theta,r} = ({X}_0,{X}_1)_{\theta,r}$ are given by
\begin{equation*}
{X}_{\theta,r} = \rB^{2\theta}_{qr,\Gamma_1}(\Omega_p) \times \bigl( \rB^{2\theta}_{qr,\Gamma_2}(\Omega_f) \cap \rL^q_{\sigma}(\Omega_f) \bigr).  
%\qquad \text{ for } 0 < \theta < \nicefrac{1}{2} \cdot \bigl(1+\nicefrac{1}{q}\bigr)\text{ and } r \in (1,\infty) .
			%			%	\begin{cases}
				%					\{p \in \rB_{qs}^{2\theta}(\Omega_p) \colon p|_{\Gamma_1} = 0 \} \times (\{ u \in \rB_{qs}^{2\theta}(\Omega_f) \colon u|_{\Gamma_2}  = 0 \}~\cap~\rL^q_{\sigma}(\Omega_f)), &\text{ if } \theta \in \left(\frac{1}{2q},\frac{1}{2} \cdot \left(1+\frac{1}{q}\right) \right) \\
				%					%
				%					\rB_{qs,00}^{\frac{1}{q}}(\Omega_p) \times ( \rB_{qs,00}^{\frac{1}{q}}(\Omega_f) ~\cap~\rL^q_{\sigma}(\Omega_f)),
				%					\textcolor{red}{?} &\text{ if } \theta = \frac{1}{2q}\\
				%					\rB_{qs}^{2\theta}(\Omega_p) \times (\rB_{qs}^{2\theta}(\Omega_f) ~\cap~\rL^q_{\sigma}(\Omega_f)), &\text{ if } \theta < \frac{1}{2q} .
				%				\end{cases} 
		\end{equation*}
	\end{enumerate} 
\end{proposition}

The following corollary summarizes the properties of the semigroup.

\begin{cor}\label{cor:sg}
The operator $\mb{A}$ generates an exponentially stable, compact, analytic $C_0$-semigroup of angle $\frac{\pi}{2}$ on $X_0$, the Beaver-Joseph-Saffman semigroup.
\end{cor}

The proof of \cref{prop:mr} is subdivided into several steps. We start with part(a).

\begin{proof}[Proof of \cref{prop:mr}(a).]
By compactness of the Sobolev embeddings we obtain
\begin{equation*}
	{X}_1 \hookrightarrow D(\Delta_m) \times D(A_m) = \rH^{2,q}(\Omega_p) \times 
	(\rH^{2,q}(\Omega_f) \cap \rL^q_{\sigma}(\Omega_f))
	\stackrel{c}{\hookrightarrow}
	\rL^q(\Omega_p) \times \rL^q_{\sigma}(\Omega_f)
\end{equation*}
and hence compactness of the resolvent of $\mb{A}$. This implies $\sigma(\mb{A}) = \sigma_p(\mb{A})$, and hence $q$-independence of the spectrum.	
Therefore, in order to prove $L^q$-spectral properties of $\mb{A}$ it suffices to determine the eigenvalues for the particular case $q = 2$. 
We consider the equations
\begin{equation*}
	\left\{
	\begin{aligned}
		\lambda p &= k\Delta_m p &&\text{ on } \Omega_p, \\
		\lambda u &= A_m u &&\text{ on } \Omega_f, 
%		\\
%		\sigma_f n_\Gamma &= - p n_\Gamma - \beta \sum_{\ell = 1}^2 (u \cdot t^\ell_\Gamma) t_\Gamma^\ell, &&\text{ on } \Gamma, \\
%		u \cdot n_\Gamma &= - k \nabla p \cdot n_\Gamma, &&\text{ on } \Gamma, \\
%		&p|_{\Gamma_1} = 0, \qquad u|_{\Gamma_2} = 0.
	\end{aligned}
	\right. 
\end{equation*}
subject to the interface conditions \eqref{eq:ns-darcyint} and the boundary conditions $p|_{\Gamma_1} = 0$, $u|_{\Gamma_2} = 0$. {\color{black} Integration by parts, the interface conditions and  Remark \ref{pressure} yield
\begin{align*}
	2 \mathrm{Re} \lambda \cdot (\| p \|_{\rL^2}^2 + \| u \|_{\rL^2}^2)
	&= \int_{\Omega_p} \lambda p \bar{p} + \int_{\Omega_p} p \bar{\lambda} \bar{p} + \int_{\Omega_f} \lambda u\cdot \bar{u} + \int_{\Omega_f} u \cdot\bar{\lambda} \bar{u} \\
	&= k \int_{\Omega_p} \Delta_m p \bar{p} + k \int_{\Omega_p} p \Delta_m \bar{p} + \int_{\Omega_f} A_m u \cdot \bar{u} + \int_{\Omega_f} u \cdot A_m \bar{u} \\ 
	&= -2k \| \nabla p \|_{\rL^2}^2 - 2\mu \| D(u) \|_{\rL^2}^2 + 2\mathrm{Re}\left[- k\int\limits_{\Gamma}(\nabla p\cdot {n}_{\Gamma})p + \int\limits_{\Gamma}\sigma {n}_{\Gamma}\cdot {u}\right]\\
	&= -2k \| \nabla p \|_{\rL^2}^2 - 2\mu \| D(u) \|_{\rL^2}^2 
	- 2 \beta \sum_{\ell = 1}^2 \int_\Gamma |u \cdot t_\Gamma^\ell|^2
	\leq 0 .
%	+ \int_{\Gamma} k \nabla p \cdot n_\Gamma \bar{p} 
%	+ \int_{\Gamma} p k \nabla \bar{p} \cdot n_\Gamma
%	+ \int_{\Gamma} \sigma_f n_\Gamma \bar{u} n_\Gamma
%	+ \int_{\Gamma} \bar{\sigma_f} n_\Gamma u n_\Gamma 
%	\\
%	&= 2k \| \nabla p \|_{\rL^2}^2 + 2 \| D(u) \|_{\rL^2}^2 
%	- \int_{\Gamma} (u \cdot n_\Gamma) \bar{p} 
%	- \int_{\Gamma} p (\bar{u} \cdot n_\Gamma)
%	- \int_{\Gamma} p n_\Gamma \bar{u} n_\Gamma 
%	- \beta \sum_{\ell=1}^2 \int_{\Gamma} (u \cdot t_\Gamma^\ell) t_\Gamma^{\ell} \cdot (\bar{u}n_\Gamma)
%	- \int_{\Gamma} \bar{p} n_\Gamma u n_\Gamma 
%	- \beta \sum_{\ell=1}^2 \int_{\Gamma} (\bar{u} \cdot t_\Gamma^\ell) t_\Gamma^{\ell} \cdot (u n_\Gamma)
\end{align*}
Here, we have used the sign of the normal vector, i.e.,  
	\begin{equation*}
		{n}_{\Gamma}={n}_f \mid_{\Gamma}=-{n}_p\mid_{\Gamma}.
	\end{equation*}
Hence, $\mathrm{Re}\lambda \leq 0$. Finally, for the case $\mathrm{Re}\lambda = 0$ we obtain that $p$ and $u$ are constant. Now the boundary conditions imply $u = p = 0$ and hence $i\R \subseteq \rho(\mb{A})$.}
\end{proof} 

The main idea of the proof of the remaining parts is to deal with the boundary terms by a perturbation argument and decouple the operator matrix with this. To this end,  we introduce the 
boundary perturbation operator 
\begin{equation*}
	{\Phi}
	\binom{p}{u} := \binom{- \frac{1}{k} u \cdot n_{\Gamma}}{-p n_{\Gamma} - \beta\sum_{\ell=1}^{2}(u\cdot t^{\ell}_{\Gamma})t^{\ell}_{\Gamma}} 
\end{equation*}
\textcolor{black}{and the Neumann boundary operator 
\begin{equation*} 	
	\rL \binom{p}{u} := \binom{\partial_{n_{\Gamma}} p}{\sigma n_{\Gamma}}.
\end{equation*} 
Note that $\rL$ represents the leading order terms of the boundary conditions \eqref{bdry:BJ}, whereas $\Phi$ collects the lower order terms. Hence, the operator $\mathbb{A}$ given in 
\eqref{def:opmatrix}-\eqref{dom:opmatrix} can be rewritten as
\begin{equation*}
	\mathbb{A} \binom{p}{u}  = \mathbb{A}_m \binom{p}{u} , \qquad 
	X_1 = \left\{ \binom{p}{u} \in D(\mathbb{A}_m) \ \big| \ \rL \binom{p}{u} = \Phi \binom{p}{u} \right\} .
\end{equation*}
}

Furthermore, we define the operators
\begin{align*}
	{Q}_0 := \rL_0 {\Phi} \qquad { \mbox{and} } \qquad {Q} := -\mb{A}_{-1}{Q}_0 ,
\end{align*}
where $\mb{A}_{-1} := (\mb{A}_0)_{-1}$ in the sense of \cref{sec:prelim} and obtain 

\begin{lemma}\label{lem:Q bounded}
If  $\gamma < \frac{1}{2} \cdot \big( 1 + \frac{1}{q} \big)$, then ${Q}_0 \in \mathcal{L}({Z}_{\frac{1}{2}},{Z}_{\gamma})$. Hence, ${Q} \in \mathcal{L}({Z}_{\frac{1}{2}},{Z}_{\gamma-1})$ for all 
$\gamma <\frac{1}{2} \cdot \big( 1 + \frac{1}{q} \big)$.
\end{lemma}

\begin{proof}
The trace theorem implies that ${\Phi}$ is bounded from
${Z}_{\frac{1}{2}} \hookrightarrow \rH^{1,q}(\Omega_p) \times (\rH^{1,q}(\Omega_f) \cap \rL^q_{\sigma}(\Omega_f))$ to $\mathrm{W}^{1-\frac{1}{q},q}(\Gamma)\times \mathrm{W}^{1-\frac{1}{q},q}(\Gamma)$. 
Now the mapping properties of $\rL_0$ described above  imply that ${Q}_0$ is bounded from ${Z}_{\frac{1}{2}}$ to $\rH^{2,q}_{\Gamma_1}(\Omega_p) \times (\rH^{2,q}_{\Gamma_2}(\Omega_f) \cap \rL^q_{\sigma}(\Omega_f))$. 
%$\{p \in \rH^{2,q}(\Omega_p) \colon p|_{\Gamma_1} = 0 \} \times (\{ u \in \rH^{2,q}(\Omega_f) \colon u|_{\Gamma_2}  = 0 \}~\cap~\rL^q_{\sigma}(\Omega_f))$.
\cref{lem:A_0}(b) yields 
	\begin{equation*} 
%		\{p \in \rH^{2,q}(\Omega_p) \colon p|_{\Gamma_1} = 0 \} \times (\{ u \in \rH^{2,q}(\Omega_f) \colon u|_{\Gamma_2}  = 0 \}~\cap~\rL^q_{\sigma}(\Omega_f))
		\rH^{2,q}_{\Gamma_1}(\Omega_p) \times (\rH^{2,q}_{\Gamma_2}(\Omega_f)\cap~\rL^q_{\sigma}(\Omega_f))
		\hookrightarrow 
		%		\rH^{2\beta,q}(\Omega_p) \times (\rH^{2\beta,q}(\Omega_f)~\cap~\rL^q_{\sigma}(\Omega_f)) =
		{Z}_{\gamma}
	\end{equation*} 
for all $0 \leq \gamma < \frac{1}{2} \cdot \big( 1 + \frac{1}{q} \big)$. Thus the claim follows by the closed graph theorem. Finally, the second claim follows immediately from the first one.
\end{proof}

In the sequel, we  rewrite $\mb{A}$ using $\mb{A}_0$ and ${Q}$ in the following way. 

\begin{lemma}\label{lem:representation A}
The operator $\mb{A}$ can be written as \textcolor{black}{the $X_0$-representation of $\mb{A}_{-\frac{1}{2}} + {Q}$, i.e. 
	\begin{align*}
		\mb{A}v &= (\mb{A}_{-\frac{1}{2}} + {Q} )v , \\
		D(\mb{A}) &= \{ v \in Z_{\frac{1}{2}} \colon (\mb{A}_{-\frac{1}{2}} + {Q} )v \in X_0 \}
	\end{align*}
}
	where $\mb{A}_{-\frac{1}{2}} := (\mb{A}_0)_{-\frac{1}{2}}$.
\end{lemma}

\begin{proof}
	\textcolor{black}{We follow the strategy from \cite[Lemma 3.4]{ABE:16}.
We denote the operator on the right hand side above by $\tilde{\mb{A}}$. %Let ${L}\binom{p}{u} = \binom{\partial_{n_{\Gamma}} p}{\sigma n_{\Gamma}}$ be the Neumann boundary operator. 
%Furthermore, we define the maximal operator $\mb{A}_m := 
%	\begin{pmatrix} 
%		k \Delta_m & 0 \\ 0 & A_m
%	\end{pmatrix}$ 
%with diagonal domain $D(\mb{A}_m) = D(\Delta_m) \times D(A_m)$. 
We consider $v \in D(\mathbb{A}) = X_1 = \left\{ \tbinom{p}{u} \in D(\mathbb{A}_m) \ \big| \ \rL \tbinom{p}{u} = \Phi \tbinom{p}{u} \right\}$. Using $\rL \rL_0 = \mathrm{Id}$ and $Q_0 = \rL_0 \Phi$ it follows 
\begin{equation*}
	(1-Q_0) v \in \ker(\mathrm{L}) = D(\mathbb{A}_0) = Z_1 .
\end{equation*}
This implies ${\mathbb{A}}_{-1}(1 - Q_0)v\in Z_0 = X_0$. 
Furthermore, using \cref{lem:A_0}(b) we obtain
$v \in D(\mathbb{A}) \subset D(\mathbb{A}_m)  \hookrightarrow {Z}_{\frac{1}{2}}$. Now, $\mathbb{A}_{-\frac{1}{2}}v = \mathbb{A}_{-1}v$ for $v \in Z_{\frac{1}{2}}$ and the definition of $Q = \mathbb{A}_{-1}Q_0$ imply
\begin{equation*}
	 {\mathbb{A}}_{-{\frac{1}{2}}}v - Qv
	= {\mathbb{A}}_{-1}v - {\mathbb{A}}_{-1}Q_0v
	= {\mathbb{A}}_{-1}(1 - Q_0)v\in X_0
\end{equation*}
and we conclude $D(\mathbb{A}) \subseteq D(\tilde{\mathbb{A}})$.
}

\textcolor{black}{
For the opposite direction let $v= (p,u)^{\top} \in D(\tilde{\mathbb{A}})$, i.e. $v \in {Z}_{\frac{1}{2}}$ with ${\mathbb{A}}_{-{\frac{1}{2}}}v - Qv\in{X}_0$.
Using $\mathbb{A}_{-\frac{1}{2}}v = \mathbb{A}_{-1}v$ for $v \in Z_{\frac{1}{2}}$ and $Q = \mathbb{A}_{-1}Q_0$ we obtain ${\mathbb{A}}_{-1}(1 - Q_0)v\in{X}_0=Z_0$. This implies
$(1 - Q_0)v\in Z_1 = \ker(\rL)$. Using
\begin{equation*} 
	v = (1 - {Q}_0)v + {Q}_0 {v} \in {Z}_1 \oplus \ker(\mb{A}_m) \subseteq D(\mb{A}_m),
\end{equation*}
we are in the position to apply the boundary operator $\rL$. Applying $\rL$ and using $\rL \rL_0 = \mathrm{Id}$ and $Q_0 = \rL_0 \Phi$ we conclude $\rL v = \Phi v$, which shows $D(\tilde{\mathbb{A}}) \subseteq D(\mathbb{A})$ and hence $D(\tilde{\mathbb{A}}) = D(\mathbb{A})$.
} 
%Using the definition of $Q=\mathbb{A}_{-1}Q_0$, we obtain for ${v} = (p,u) \in {Z}_{\frac{1}{2}}$ that ${v} \in D(\tilde{\mb{A}})$ holds if and only if 
%${\mathbb{A}}_{-1}(1 - Q_0)v\in{X}_0=Z_0$, since $\mb{A}_{-1} v = \mb{A}_{-\frac{1}{2}} v$ for $v \in {Z}_{\frac{1}{2}}$. The latter statement is equivalent to 
%$(1-Q_0) v \in {Z}_1 \subset \ker({L})$. Using 	${L}{L}_0 = \mathrm{Id}$, the definition of $Q_0$ and 
%\begin{equation*} 
%	v = (1 - {Q}_0)v + {Q}_0 {v} \in {Z}_1 \oplus \ker(\mb{A}_m) \subseteq D(\mb{A}_m) \hookrightarrow {Z}_{\frac{1}{2}} ,
%\end{equation*}
%where the last inclusion follows from \cref{lem:A_0}(b),
%$(1-{Q}_0){v} \in {Z}_1 \subset \ker({L})$ is equivalent to ${L}{v} = {\Phi}{v}$.
%Finally, the latter statement is true if and only if $v \in D(\mb{A})$.
%	
%	\begin{align*}
%		\mb{v} \in D(\tilde{\mb{A}})
%		\quad &\Longleftrightarrow \quad 
%		(\mb{A}_{-{\frac{1}{2}}}-1)(1 - \mb{Q}_1) \mb{v} + \mb{v} \in \mb{X}_0 \\
%		&\Longleftrightarrow \quad (1-\mb{Q}_1) \mb{v} \in \mb{Z}_1 \subset \ker(\mb{L}) \\
%		&\Longleftrightarrow \quad \mb{L}\mb{v} = \mb{\Phi}\mb{v}  \\
%		&\Longleftrightarrow \quad \mb{v} \in D(\mb{A}) , 
%	\end{align*}
%	where we used $\mb{A}_{-1}|_{\mb{Z}_{-\frac{1}{2}}} = \mb{A}_{-\frac{1}{2}}$ and $\mb{L} \mb{L}_1 = \mathrm{Id}$, as well as, 
%	\begin{equation*} 
%		\mb{v} = (1 - \mb{Q}_1)\mb{v}+\mb{Q}_1 \mb{v} \in \mb{Z}_1 \oplus \ker(1-\mb{A}_m) \subseteq D(\mb{A}_m) 
%		\hookrightarrow \mb{Z}_{\frac{1}{2}} ,
%	\end{equation*} 
%	where the last inclusion follows from \cref{lem:A_0}(b).

\textcolor{black}{
	Furthermore, for ${v} \in D(\mb{A}) \subset Z_{\frac{1}{2}} \subset  D(\mb{A}_m)$ we obtain from $\mb{A}_{-\frac{1}{2}}v= \mb{A}_m v = \mb{A} v$ and 
$\mb{A}_m{Q}_0{v} = 0$ that
	\begin{align*}
		\tilde{\mb{A}}{v} = \mb{A}_m (1- {Q}_0){v} = \mb{A}_m {v} = \mb{A} v 
	\end{align*}
	and hence the claim.}
\end{proof}

\begin{proof}[Proof of {\cref{prop:mr}(b) and (c)}.]
(b): Using \cref{lem:Q bounded} and the representation of $\mb{A}$ as given in \cref{lem:representation A}, we conclude from 
\cref{thm:boundary perturbation mr} with  $\beta < \frac{1}{2} \cdot \big( 1 + \frac{1}{q} \big)$ that $\mb{A}+\lambda$ is $\mathcal{R}$-sectorial on ${X}_0$ for all $\lambda > s(\mb{A})$. 
From \cref{prop:mr}(a) we know that $s(\mb{A}) < 0$ and hence we may choose $\lambda = 0$. 
	
(c): From \cref{lem:Q bounded} we conclude that ${Q} \in \mathcal{L}({Z}_{\frac{1}{2}},{Z}_{\beta-1})$ for $\beta \in [\frac{1}{2},\frac{1}{2} \cdot \big( 1 + \frac{1}{q} \big))$. 
It follows from \cite[Section 4]{HHK:04} that  
	\begin{equation*}
		{X}_{\theta,r} = {Z}_{\theta,r} 
	\end{equation*} 
	for $\theta < \beta < \big( 1 + \frac{1}{q} \big)$ and $r \in (1,\infty)$. Finally, the claim follows from \cref{lem:A_0}(c).
\end{proof}

\begin{proof}[Proof of {\cref{cor:sg}}.]
The generation of an analytic $C_0$-semigroup follows immediately from the maximal regularity. Hence, the compactness of the resolvent implies compactness of the semigroup. Further, the exponential stability follows from $s(\mb{A}) < 0$. It remains to calculate the angle of analyticity. Note that $\mb{A}_0$ generates an analytic $C_0$-semigroup of angle $\frac{\pi}{2}$. Now the claim follows from \cref{prop:sw} %\cite[Theorem 2.4]{ABE:16} 
	using \cref{lem:Q bounded} and \cref{lem:representation A} similar to the proof of \cref{prop:mr}(a). 
\end{proof}

\subsection{The linearized equation subject to Beaver-Joseph boundary conditions}
\

In the sequel we analyze the linearized problem subject to the  Beavers-Joseph interface condition \eqref{bdry:BJ} for a sufficiently small  coefficient $\beta > 0$. 
%Actually $\beta$ has a scaling like square root of the porosity. So, if the porosity of the medium is small (which is true for most porous media, see ....), the smallness assumption of coefficient $\beta$ is physically relevant.

%Now, we can rewrite equations \eqref{eq:ns-darcyp}--\eqref{eq:ns-darcyint}$_2$ with Beavers-Joseph condition \eqref{bdry:BJ} as the abstract semilinear Cauchy problem \eqref{eq:acp}. 
To this end, we consider the Hilbert space ${X}_0= \rL^2(\Omega_p)\times \rL^2_{\sigma}(\Omega_f)$:
Observe that the interface conditions \eqref{eq:ns-darcyint}$_{1-2}$ and \eqref{bdry:BJ} imply
\begin{equation*}
	\sigma {n}_{\Gamma}=({n}_{\Gamma}\cdot\sigma {n}_{\Gamma}){n}_{\Gamma}+\sum_{\ell=1}^{2}(t^{\ell}_{\Gamma}\cdot \sigma_f{n}_{\Gamma})t^{\ell}_{\Gamma} = -p{n}_{\Gamma} - \beta\sum_{\ell=1}^{2}(({u}-\nabla p)\cdot 
t^{\ell}_{\Gamma})t^{\ell}_{\Gamma}.
\end{equation*}
This motivates the definition of the Beaver-Joseph operator $\mb{A}_{BJ}:{X}_1 \subset {X}_0\rightarrow {X}_0$ as 
\begin{equation}\label{def:ABJ}
	\mb{A}_{BJ}=\begin{pmatrix}
		k \Delta_m & 0 \\  0 & A_m
	\end{pmatrix}
\end{equation}
with domain
\begin{equation}\label{dom:ApfBJ}
	{X}_1=D(\mb{A}_{BJ})=\left\{\binom{p}{u}\in D(\Delta_m)\times D(A_m)\left|
	\begin{matrix} 
		\sigma n_{\Gamma} = -pn_{\Gamma} - \beta\sum_{\ell=1}^{2}((u-\nabla p)\cdot t^{\ell}_{\Gamma})t^{\ell}_{\Gamma}, \\
		u\cdot n_{\Gamma}=-k\nabla p\cdot n_{\Gamma}, \quad 
		p|_{\Gamma_1} = 0, \quad 
		u|_{\Gamma_2} = 0.
	\end{matrix} \right. \right\},
\end{equation}
where $D(\Delta_m)= \rH^{2}(\Omega_p)$ and $D(A_m)= \rH^{2}(\Omega_f) \cap \rL^2_{\sigma}(\Omega_f)$.

Let us consider $\theta <\frac{\pi}{2}$. We then obtain the following lemma.

\begin{lemma}\label{lem:dissipative}
 Let $\theta <\frac{\pi}{2}$ and assume $\beta$ to be sufficiently small. Then the operators $(e^{i\vartheta}\mb{A}_{\mathrm{BJ}})$ are dissipative for all $\vartheta \in (-\theta,\theta)$.
\end{lemma}

\begin{proof}
For $(p,{u})\in D(\mb{A}_{BJ})$ we calculate
\begin{equation*}
	\begin{aligned}
		2 \mathrm{Re}	\left\langle e^{ i\vartheta}\mb{A}_{BJ}(p,{u}), (p,{u})\right\rangle_{{X}_0}
		&= \left\langle e^{ i\vartheta}(k\Delta_m p, A_m {u}), (p,{u})\right\rangle_{{X}_0} + \left\langle (p,{u}),e^{ i\vartheta}(k\Delta_m p, A_m {u}) \right\rangle_{{X}_0}
		\\
		&= e^{ i\vartheta}\Big(\int\limits_{\Omega_p} k\Delta_m p\cdot \bar{p} + \int\limits_{\Omega_f} A_m{u}\cdot \bar{u}\Big)
		+ e^{-i\vartheta}\Big(\int\limits_{\Omega_p} k\Delta_m \bar{p}\cdot p + \int\limits_{\Omega_f} A_m \bar{u}\cdot {u}\Big)
		\\ &= -2 \cos \vartheta \Big(k\int\limits_{\Omega_p}|\nabla p|^2 + 2\mu \int_{\Omega_f} |D(u)|^2 \big) \\ %+ 2 k \cdot \mathrm{Re} \int\limits_{\Gamma}e^{i\theta}(\nabla p\cdot\mb{n}_p)\bar{p} \\
		& \quad - \beta\sum_{\ell=1}^{2}\int\limits_{\Gamma}|{u}\cdot t^{\ell}_{\Gamma}|^2 + \beta\sum_{\ell=1}^{2} \mathrm{Re} \int\limits_{\Gamma} e^{i\vartheta} ({u}\cdot t^{\ell}_{\Gamma})(\nabla \bar{p}\cdot t^{\ell}_{\Gamma})
	\end{aligned}
\end{equation*}
%We can interpret the last boundary term in the following way:
%\begin{equation*}
%	\int\limits_{\Gamma_I}(\bar{u}\cdot t^{\ell}_{\Gamma})(\nabla p\cdot t^{\ell}_{\Gamma}) = \langle\nabla p, {u}\rangle_{H^{-1/2}(\Gamma_I),H^{1/2}(\Gamma_I)}
%\end{equation*}
Note that  
\begin{align*}
	|\beta\int\limits_{\Gamma_I}({u}\cdot t^{\ell}_{\Gamma})(\nabla \bar{p}\cdot t^{\ell}_{\Gamma})| = \beta|\langle\nabla p, {u}\rangle_{H^{-1/2}(\Gamma_I),H^{1/2}(\Gamma_I)}|&\leq \beta\cdot\|p\|_{H^1(\Omega_p)}\|{u}\|_{H^1(\Omega_f)} \\
	&\leq \frac{\beta}{2} \cdot (\|p\|^2_{H^1(\Omega_p)}+\|{u}\|^2_{H^1(\Omega_b)}).
\end{align*}
Hence, if $\vartheta\in (-\theta,\theta)$ and $\beta\geq 0$ is small enough, we may conclude that $e^{ i\vartheta}\mb{A}_{\mathrm{BJ}}$ is a dissipative operator in $\mb{X}_0$.
\end{proof}
% Thus we have the following result following the steps of \cref{thm:op}:

Next, we would like  to show that the operator $\lambda_0 I - e^{ i\vartheta}\mb{A}_{\mathrm{BJ}} \colon D(\mb{A}_{\mathrm{BJ}}) \subset {X}_0 \to {X}_0$ is surjective for some $\lambda_0$. 
In the following, we show that this is true in particular for $\lambda_0=1$.

\begin{lemma}\label{lem:maximal}
{\color{black}Let $\theta <\frac{\pi}{2}$. The operator $1-e^{i\vartheta}\mb{A}_{\mathrm{BJ}}$ is surjective for all $\vartheta \in (-\theta,\theta)$.}
\end{lemma}

\begin{proof} 
Given $(f_1, f_2)\in {X}_0$, we need to find  $(p,u)\in {D}(\mb{A}_{\mathrm{BJ}})$ satisfying the following equations
	\begin{equation}\label{eq:max1}
			\left\{
			\begin{aligned} 
					p - e^{ i\vartheta} k \Delta p&=f_1,\\
					u - e^{ i\vartheta} A_m u&=f_2.
				\end{aligned}
			\right. 
		\end{equation}
Consider the space
	\begin{equation*}
			{Y}:= \rH^1_{\Gamma_1}(\Omega_p) \times \bigl( \rH^1_{\Gamma_2}(\Omega_f) \cap \rL^2_{\sigma}(\Omega_f) \bigr) 
			%\left\{(p,u)\in H^1(\Omega_p)\times \mb{H}^1_{\sigma}(\Omega_f)\mid\ u=0\ \mbox{on}\ \Gamma_f^{ext}\right\}
		\end{equation*}
	equipped with the inner product
	\begin{align*}
			\left\langle(p_1,u_1),(p_2,u_2)\right\rangle_{{Y}}:=&
			\int\limits_{\Omega_p} p_1\bar{p}_2
			+  \int\limits_{\Omega_p} \nabla p_1 \cdot \nabla \bar{p}_2 
			+\int\limits_{\Omega_f} u_1\cdot \bar{u}_2
			+2 \mu \int\limits_{\Omega_f} D(u_1):D(\bar{u}_2).
		\end{align*}
	{\color{black} Motivated by the interface conditions}, we define the sesquilinear linear form $a: {Y}\times {Y}\rightarrow \mathbb{C}$ as
	\begin{align*}
			a((p,u), (\xi,\zeta))&:=  \int\limits_{\Omega_p} p\cdot\bar{\xi}+ \int\limits_{\Omega_f} u\cdot\bar{\zeta} 
			+ e^{ i\vartheta}\biggl(k\int\limits_{\Omega_p}\nabla p\cdot\nabla\bar{\xi} + 2\mu \int\limits_{\Omega_f} D(u):D(\bar{\zeta}) \biggr) \\
			& \qquad + e^{i \vartheta}\biggl(\langle p{n}_{\Gamma},{\zeta} \rangle_{\Gamma}+  \beta\sum_{\ell=1}^{2}\langle u\cdot t^{\ell}_{\Gamma},\zeta\cdot t^{\ell}_{\Gamma}\rangle_{\Gamma}+ \langle -u\cdot{n}_{\Gamma},\xi\rangle_{\Gamma}\biggr).
		\end{align*}
Since, 
\begin{align*}
\mathrm{Re}\ a((p,u), (p,u)) &= \int\limits_{\Omega_p} |p|^2 + \cos\vartheta \cdot \left (k\int\limits_{\Omega_p}|\nabla p|^2+\int\limits_{\Omega_f} |u|^2+ 2\mu \int\limits_{\Omega_f} |D(u)|^2+
\beta\sum_{\ell=1}^{2}\int\limits_{\Gamma}|u\cdot t^{\ell}_{\Gamma}|^2 \right) \\ & \geq {\color{black}c(\vartheta)\|(p,u)\|_{{Y}}^2},
\end{align*}
for $(p,u)\in {Y}$, it follows that the sesquilinear linear form $a$ is coercive. Moreover, $a$ is also continuous on ${Y}\times {Y}$.
Furthermore, consider the mapping $L:{Y}\rightarrow \mathbb{C}$ given by
	\begin{equation*}
			L((f_1, f_2),(\xi,\zeta)):= \int\limits_{\Omega_p} f_1 \cdot \bar{\xi} + \int\limits_{\Omega_f} f_2\cdot \bar{\zeta} .
		\end{equation*}
Obviously, the mapping $L$ is linear and continuous on ${Y}$. Applying the Lax-Milgram theorem we  conclude that there exists a unique weak solution to \eqref{eq:max1}.
It follows that $(u,p) \in \rH^1(\Omega_{p}) \times \rH^1(\Omega_f)$ and satisfies \eqref{eq:max1} and \eqref{eq:ns-darcyint} in the sense of distributions. 
%	
%	
%	
%	Hence, we have $(p,u) \in H^1(\Omega_{p}) \times {H}^1(\Omega_f)$ with 
%	\begin{equation*}
%			\left\{
%			\begin{aligned}
%					p-e^{ i\vartheta}A_p p  &= f_1 \in L^2(\Omega_{p}) \quad \text{ in the sense of distributions } \\
%					u-e^{ i\vartheta}A_f u  &= \mb{f}_2 \in L^2(\Omega_{f}) \quad \text{ in the sense of distributions }
%				\end{aligned}
%			\right. 
%		\end{equation*}
%	and 
%	\begin{equation*} 
%			\left\{
%			\begin{aligned}
%					u\cdot {n}_{\Gamma}&=-k\nabla p\cdot {n}_{\Gamma},\\
%					{n}_{\Gamma}\cdot\sigma_f{n}_{\Gamma}&=-p,\\
%					{t}^{j}_{\Gamma}\cdot\sigma_f{n}_{\Gamma}&=-\beta u\cdot {t}^{j}_{\Gamma},\quad 1\leq j \leq 2 .
%				\end{aligned}
%			\right.
%		\end{equation*}
	The trace theorem implies that $-p \in \rH^{1/2}(\Gamma)$ and $u \in \rH^{1/2}(\Gamma)$. It follows $u \cdot {n}_\Gamma, u \cdot {t}_\Gamma^j \in \rH^{1/2}(\Gamma)$. Therefore we obtain an inhomogeneous Neumann problem for the Laplacian, given by
	\begin{equation}
			\left\{
			\begin{aligned} 
					p-k \Delta p  &= f_1 &&\text{ on } \Omega_p, \\
					- k \nabla p\cdot {n}_{\Gamma} &= g_1 &&\text{ on } \Gamma,
				\end{aligned} 
			\right.
			\label{eq:porous}
		\end{equation}
	with $f_1 \in \rL^2(\Omega_{p})$ and $g_1 := u\cdot {n}_{\Gamma} \in \rH^{1/2}(\Gamma)$,
	and a Stokes problem, given by
	\begin{equation} 
\left\{
\begin{aligned}
-A_m u  &= f_2 &&\text{ on } \Omega_f, \\
\sigma {n}_{\Gamma} &=  -p{n}_{\Gamma}-\beta\sum_{\ell=1}^{2}(u\cdot t^{\ell}_{\Gamma})t^{\ell}_{\Gamma}= g_2, &&\text{ on } \Gamma, \\
u &= 0 &&\text{ on } \Gamma_2,
				\end{aligned} 
			\right. 
			\label{eq:fluid2}
		\end{equation}
with $f_2 \in \rL^2(\Omega_{f})$ and $g_2  \in \rH^{1/2}(\Gamma)$. We know e.g. from \cite[Theorem III.4.3]{MR2986590} that the solution $p$ of \eqref{eq:porous} satisfies  
$p \in \rH^2(\Omega_p)$. Finally, \cite[Theorem IV.7.1]{MR2986590} yields that the solution $u$ of \eqref{eq:fluid2} satisfies  $u \in \rH^2(\Omega_f)$. 
Summing up, we  conclude that $(p,u)\in {D}(\mb{A}_{\mathrm{BJ}})$.
\end{proof} 

Given Lemma \ref{lem:dissipative} and Lemma \ref{lem:maximal} it is now not difficult anymore to show that the Beaver-Joseph operator generates an analytic semigroup on $X_0$.  

\begin{proposition}\label{thm:opBJ}
Assume that $\beta > 0$ is small enough and let $\mb{A}_{BJ}$ with domain $D(\mb{A}_{BJ})$ be given as in  \eqref{def:ABJ}--\eqref{dom:ApfBJ}. Then 
$(\mb{A}_{BJ},D(\mb{A}_{BJ}))$ generates a bounded, compact, analytic semigroup on ${X}_0$, the Beaver-Joseph semigroup. In addition, $\mb{A}_{BJ}$ has the property of maximal $\rL^2$-regularity on ${X}_0$. 
\end{proposition}

\begin{proof}
	%We know from \cite[Chapter 2, Section 4, Theorem 4.6]{EN00} that $\mb{A}$ generates a bounded analytic semigroup on $X_0$ iff there exists $\theta\in (0,\frac{\pi}{2})$ such that the 
%operators $e^{\pm i\theta'}\mb{A}$ generate a bounded strongly continuous semigroups on $\mb{X}_0$ for all $\theta' \in (0,\theta)$. 
The  Lumer-Phillips theorem implies together with  \cref{lem:dissipative} and \cref{lem:maximal} that $e^{\pm i\vartheta}\mb{A}: D(\mb{A})\rightarrow {X}_0$ generate contraction semigroups on $X_0$.
Now \cite[Cor 3.9.9]{ABHN:11} implies the claim. 
\end{proof}

\section{Strong well-posedness}
	\label{sec:wellposedness}

%Note that the non-linearity is the same as the non-linearity of the Navier-Stokes equations. Hence, the proof of the main result follows the same lines as the proof for the Navier-Stokes equations by some mod%ifications due to the interface conditions and the coupling. Therefore, we will be brief in the sequel and address only the modifications in detail.

We rewrite  the equations \eqref{eq:ns-darcyp}--\eqref{eq:ns-darcyint} as a semilinear evolution equation of the form 
\begin{equation}\label{eq:acp}
	v'-\mb{A} {v}
	= {F} ({v},{v}) + {f},
	\quad {v}(0)=
	\binom{p_0}{u_0},
\end{equation}
for the unknown ${v} := (p,u)^T$ with forcing  term $f := (f_p,f_f)^T$, and where the nonlinear term ${F}$ is defined as
\begin{equation*}
{F}(v_1,v_2) 
:=\begin{pmatrix}
		0 \\ -\mathcal{P} (u_1 \cdot \nabla u_2) 
	\end{pmatrix},
\end{equation*}
for ${v}_i := (p_i,u_i)$ with $i \in \{1,2\}$. For brevity, we set ${F}({v}) := {F}({v},{v})$. Furthermore, we recall the maximal regularity spaces 
\begin{align*}
 \mathbb{E}_{1,\mu}(T) &:= \rH^{1,r}_{\mu}(0,T;{X}_0) \cap  \rL^r_{\mu}(0,T;{X}_1) \mbox{ and } \\
_0 \mathbb{E}_{1,\mu}(T) &:= \{{v}\in\rH^{1,r}_{\mu}(0,T;{X}_0) \cap  \rL^r_{\mu}(0,T;{X}_1) \mid {v}(0) = 0 \}, 
\end{align*}
the data space $\E_{0,\mu}(T) := \rL^r_\mu(0,T;{X}_0)$ as well as the trace space ${X}_{\gamma,\mu} := [{X}_0,{X}_1]_{\mu-\frac{1}{r},r}$. 
For the critical weight $\mu_c := \frac{3}{2q}-\frac{1}{2}+\frac{1}{r}$ we know from \cref{prop:mr}(c) that the trace space is given by
\begin{equation}
	X_{\gamma,\mu_c}
	= \rB_{q,r,\Gamma_1}^{\frac{3}{q}-1}(\Omega_p)\times \big(\rB_{q,r,\Gamma_2}^{\frac{3}{q}-1}(\Omega_f) \cap \rL^q_{\sigma}(\Omega_f)\big) .
	\label{eq:trace spaces}
\end{equation}
	
%\textcolor{black}{In the sequel we follow the strategy of Pr\"{u}ss and Wilke \cite{PW:17}.}
By H\"older's inequality and Sobolev embeddingswe obtain the following estimate on the Navier-Stokes nonlinearity. 

\begin{lemma}\label{lem:F}
Let $q \in (1,3)$ and $s = \frac{1}{2}\left(1+ \frac{3}{q} \right)$.
	Then the bilinear mapping  
	\begin{equation*} 
		G \colon \rH^{s,q}(\Omega_f) \times  \rH^{s,q}(\Omega_f) \to \rL^q(\Omega_f) \colon (u_1,u_2) \mapsto u_1 \cdot \nabla u_2
	\end{equation*} 
satisfies
	\begin{equation*}
		\|G(u_1,u_1)-G(u_2,u_2)\|_{\rL^{q}(\Omega_f)} \leq C \cdot \bigl( \| u_1 \|_{\rH^{s,q}(\Omega_f)} + \| u_2 \|_{\rH^{s,q}(\Omega_f)} \bigr) \cdot \| u_1 - u_2 \|_{\rH^{s,q}(\Omega_f)} 
	\end{equation*}
	for all $u_1, u_2 \in \rH^{s,q}(\Omega_f)$.
\end{lemma}

\begin{proof}
	\textcolor{black}{The proof follows the lines of the work of Pr\"{u}ss and Wilke \cite{PW:17}.}
H\"{o}lder's inequality implies
\begin{equation*}
\| (u_1\cdot\nabla)u_2 \|_{L^q(\Omega_f)}
		\leq \| u_1 \|_{\rL^{qm'}(\Omega_f)} \cdot\| \nabla u_2\|_{\rL^{qm}(\Omega_f)}
		\leq \| u_1 \|_{\rL^{qm}(\Omega_f)} \cdot
		\| u _2\|_{\rH^{1,qm'}(\Omega_f)} 
	\end{equation*}
	for $1/m+1/m' = 1$. {\color{black}We choose $\frac{3}{qm} = \frac{1}{2} \cdot \left( 1 + \frac{3}{q} \right)$ with $q \in (1,3)$.} Using the Sobolev embeddings
	\begin{equation*}
		\rH^{s,q}(\Omega_f) \hookrightarrow \rL^{qm}(\Omega_f) 
		\quad \text{ and } \quad \rH^{s,q}(\Omega_f) \hookrightarrow \rH^{1,qm'}(\Omega_f)  
	\end{equation*}
	for $s = \frac{1}{2} \left( 1 + \frac{3}{q} \right)$ we  obtain
	\begin{equation*}
		\| (u_1 \cdot \nabla) u_2 \|_{\rL^q(\Omega_f)}
		\leq C \cdot \| u_1 \|_{\rH^{s,q}(\Omega_f)} \cdot
		\| u _2\|_{\rH^{s,q}(\Omega_f)} .
	\end{equation*} 
	Using the bilinear structure of $G$, i.e.
	\begin{equation*}
		G(u_1,u_1)-G(u_2,u_2) = G(u_1,u_1-u_2) + G(u_1-u_2,u_2)
	\end{equation*}
	the claim follows from the previous estimate.
\end{proof}

\begin{lemma}\label{lem:Xbeta}
We have $X_\beta := [X_0,X_1]_\beta \hookrightarrow \rH^{2\beta,q}(\Omega_p) \times \bigl( \rH^{2\beta,q}(\Omega_f) \cap \rL^q_{\sigma}(\Omega_f) \bigr)$ for $\beta \in (0,1)$.
\end{lemma}

\begin{proof}
 The embedding $\iota \colon X_1 \hookrightarrow \rH^{2,q}(\Omega_p) \times \bigl( \rH^{2,q}(\Omega_f) \cap \rL^q_{\sigma}(\Omega_f) \bigr)$ is bounded. Moreover,
 $\iota \colon X_0 \hookrightarrow \rL^{q}(\Omega_p) \times \rL^q_{\sigma}(\Omega_f)$ is the identity and hence a bounded operator. 
Now the claim follows by applying the complex interpolation functor to the embedding $\iota$. 
\end{proof}

We now give a proof of our first main result concerning local strong wellposedness and global strong wellposedness for small data for the system \eqref{eq:ns-darcyp}--\eqref{eq:ns-darcyint}.

\begin{proof}[Proof of \cref{thm:main bjs critical}.]
We verify Assumption (A) in \cref{sec:prelim}. 
%From the proof of \cref{thm:local} we see that in conditions (H2) and (S) it is not necessary to have the spaces $X_\beta$ but of course the spaces used in (H2) and (S) have to coincide. Here we use the spaces $\rH^{s,q}(\Omega_p) \times (\rH^{s,q}(\Omega_f) \cap \rL^q_{\sigma}(\Omega_f))$.% for $s = \frac12(1 + 3/q)$.
Let us note first that condition  (MR) is proven in \cref{prop:mr}(b).  
To verify condition (H1) we set $\beta=\frac14(1 + 3/q)$ and use Lemmas \ref{lem:F} and \ref{lem:Xbeta}.  Condition (H2) means $2\beta \leq \mu+1-1/r$  and the optimal choice of $\mu=\mu_c$ is 
given by $\mu_c=2\beta -1+1/r$. Furthermore, the condition $\mu\leq 1$ requires $2/r + 3/q \leq 3$. 
%With $\mu_c = \frac{1}{r} + \frac{3}{2q}-\frac{1}{2}$, \cref{prop:mr}(b), t
Finally, we see from the proof of \cref{thm:local} that the condition (S) can be weaken to a mixed derivative theorem for Bessel potential spaces without boundary conditions. 
Such a mixed derivative theorem is proven in \cite[Proposition~3.2]{MS12}. Now the claim follows from \cref{thm:local}.
\end{proof}

Next, we prove the finite-in-time blow up criterion. 

\begin{proof}[Proof of {\cref{cor:bjs blow up}}.]
Claim (ii) follows the lines of the proof of \eqref{eq:serrin}, %\cite[Theorem 2.4]{PSW18}, 
where we replace the abstract mixed derivative theorem from condition (S) %\cite[Proposition 7.1]{PSW18} 
by its analogue for Bessel potential spaces without boundary conditions from \cite{MS12} and use \cref{prop:mr} and \cref{lem:F}.
Claim (i) follows immediately from \cref{sec:prelim} and ~\eqref{eq:trace spaces}.
\end{proof}

	Finally, we show the local strong wellposedness and global strong wellposedness for small data for the Navier-Stokes-Darcy system with Beaver-Joseph interface conditions. 
\begin{proof}[Proof of \cref{cor:bj r=q=2}.]
	The proof follows the lines of the proof of \cref{thm:main bjs critical} with $r = q = 2$ replacing \cref{prop:mr} by \cref{thm:opBJ}.
\end{proof}

\section{Higher regularity}
\label{sec:regularity}

This section is dedicated to the proofs of \cref{thm:reg in bjs} and \cref{thm:reg bjs}. Both rely on Angenent's parameter trick, see \cite{Ang:90}. We need to assume that $\Gamma$ is analytic.   

We start with a variant of Lemma \ref{lem:F}. With $\mu_c=1/r+3/2q-1/2$, \cref{prop:mr}(b), the mixed derivative theorem  and Sobolev embeddings imply
\begin{equation}
\begin{aligned} 
			_0 \mathbb{E}_{1,\mu}(T')
			&\hookrightarrow 
			\rL^r_{\mu}(0,T',{X}_1) \cap \,_0\rH^{1,r}_{\mu}(0,T',{X}_0) \\
			&\hookrightarrow
\rL^r_{\mu}(0,T',\rH^{2,q}(\Omega_p) \times (\rH^{2,q}(\Omega_f) \cap \rL^q_{\sigma}(\Omega_f)) ) \cap \rH^{1,r}_{\mu}(0,T';\rL^q(\Omega_p) \times \rL^q_{\sigma}(\Omega_f)) \\ 
&\hookrightarrow \rH_{\mu}^{1-s/2,r}(0,T';\rH^{s,q}(\Omega_p) \times (\rH^{s,q}(\Omega_f) \cap \rL^q_{\sigma}(\Omega_f))) \\
			&\hookrightarrow
			\rL_{\tau}^{2r}(0,T';\rH^{s,q}(\Omega_p) \times (\rH^{s,q}(\Omega_f) \cap \rL^q_{\sigma}(\Omega_f)))
\end{aligned} 
\label{eq:embedding}
\end{equation}
for $s = \frac{1}{2} \big( 1+ \frac{3}{q} \big)$ and $1-\mu = 2(1-\tau)$.
Combining Lemma \ref{lem:F} with the above embedding \eqref{eq:embedding} , we observe  that the bilinearity $F\colon \E_{1,\mu} \times \E_{1,\mu} \to \E_{0,\mu}$ is bounded. 

\begin{lemma}\label{cor:F bounded}
Let $q \in (1,3)$ and $\mu \in [\mu_c,1]$ with $\mu_c := \frac{1}{r}+\frac{3}{2q}-\frac{1}{2}$. Then the bilinearity ${F} \colon \E_{1,\mu} \times \E_{1,\mu} \to \E_{0,\mu}$ is bounded, i.e. 
there exists a constant $C > 0$ such that 
\begin{equation*}
		\| {F}({v}_1,{v}_2) \|_{\E_{0,\mu}}
		\leq C \cdot \| {v}_1 \|_{\E_{1,\mu}} \cdot \| {v}_2 \|_{\E_{1,\mu}} 
	\end{equation*}
	for all $v_1,v_2 \in \E_{1,\mu}$. 
\end{lemma}

\textcolor{black}{We are now} in the position to prove \cref{thm:reg in bjs}. 
\textcolor{black}{We follow the strategy from \cite[Section 9.4]{PS16}.}

\begin{proof}[Proof of \cref{thm:reg in bjs}.]
	We consider a solution $\bar{{v}} := (\bar{u},\bar{p})$ of \eqref{eq:ns-darcyp}-\eqref{eq:ns-darcyint} and the linear
	problem
	\begin{equation}
		\left\{
		\begin{aligned}
			\partial_t p - k\Delta p & = f_1, \\
			\partial_t u - A_m u + \mathcal{P} (\bar{u} \cdot \nabla u + u \cdot \nabla \bar{u}) & = f_2, \\
		\end{aligned}
		\label{eq:linear problem}
		\right.
	\end{equation}
	subject to  the interface conditions \eqref{eq:ns-darcyint}, the boundary conditions $p|_{\Gamma_1} = 0$, $u|_{\Gamma_2} = 0$, and the initial data $p(0) = 0$, $u(0) = 0$. 
%	\begin{equation*}
%	%	\begin{aligned} 
%			\mb{A}_{\bar{\mb{v}}} := 
%			(D_v \mb{H})(0,\bar{\textbf{v}})
%			\colon \E_{1,\mu} \to \E_1 \times \mb{X}_{\gamma,\mu}, \  
%			\mb{v} \mapsto 
%			\begin{pmatrix}
%				\partial_t p - k\Delta p \\
%				\partial_t u - A_m u + \mathcal{P} (\bar{u} \cdot \nabla u + u \cdot \nabla \bar{u}), \\
%				p(0) \\
%				u(0)
%			\end{pmatrix} .
%	%	\end{aligned} 
%	\end{equation*}
%	Note that the interface conditions are incorporated in the space $\E_{1,\mu}$.
	We denote the corresponding operator by $\mb{A}_{\overline{v}}$ and the perturbation by 
	\begin{equation*}
		%\mb{B}_{\overline{v}}(t) {v}
	\mb{B} {u}	:= \binom{0}{\mathcal{P} (\bar{u} \cdot \nabla u + u \cdot \nabla \bar{u})} .
\end{equation*}
Furthermore, we define
	\begin{equation*}
		\mb{A}_{\eta} := \mb{A} - \eta \mb{B}
	\end{equation*}
	for $\eta \in [0,1]$. Note that $\mb{A}_0 = \mb{A}$ and $\mb{A}_1 = \mb{A}_{\bar{{v}}}$.
	Now, we consider 
	\begin{equation}
	%	\left\{
	%	\begin{aligned}
			v' - \mb{A}_{\eta} {v} = {f}, \enspace {v}(0) = 0 .  
%			\\
%			\mb{v}(0) &= \mb{v}_0 := \binom{p_0}{u_0} .
	%	\end{aligned}
	%	\right. 
		\label{eq:DH}
	\end{equation}
%	We consider the reference solution 
%	\begin{equation*}
%		\mb{v}_* (t) := T_{\mb{A}}(t) \mb{v}_0
%		+ \int_0^t T_{\mb{A}}(t-s) \mb{f}(s) \mathrm{d} s .
%	\end{equation*}
%	Now $\tilde{\mb{v}} := \mb{v}-\mb{v}_*$ satisfies
%	\begin{equation*}
%		\left\{
%		\begin{aligned}
%			\mb{A}_{\eta} \tilde{\mb{v}} &= \tilde{\mb{f}}, \\
%			\mb{v}(0) &= 0 ,
%		\end{aligned}
%		\right.
%	\end{equation*}
%	where
%	\begin{equation*}
%		\tilde{\mb{f}} = \eta \mb{B} \mb{v}_* .
%	\end{equation*}
%	This is equivalent to 
%	\begin{equation*}
%		\left\{
%		\begin{aligned}
%			\mb{A} \tilde{\mb{v}} &= \tilde{\mb{f}}-\eta\mb{B} \tilde{\mb{v}}, \\
%			\mb{v}(0) &= 0 ,
%		\end{aligned}
%		\right.
%	\end{equation*}
	Using the maximal regularity of $\mb{A}$, see \cref{prop:mr}(b), and \cref{cor:F bounded}, we obtain
	\begin{equation*}
		\| {v} \|_{\E_{1,\mu}}
		\leq C_{\mathrm{MR}}  \cdot (\| {f} \|_{\E_{0,\mu}} + \eta \| \mb{B}{{v}} \|_{\E_{0,\mu}})
		\leq C_{\mathrm{MR}}  \cdot \| {f} \|_{\E_{0,\mu}} + C_{\mathrm{MR}} \cdot C \cdot \| \bar{{v}} \|_{\E_{1,\mu}} \cdot 
		\| {{v}} \|_{\E_{1,\mu}} .
	\end{equation*}
	Now we choose $T_1 < T'$ sufficient small such that 
	\begin{equation}
		C_{\mathrm{MR}} \cdot C \cdot \| \bar{{v}} \|_{\E_{1,\mu}}
		\leq \frac{1}{2} \label{eq:vbar}
	\end{equation}
	holds. We point out that the constants $C_{\mathrm{MR}}, C > 0$ are independent of $T'$ and $\eta$. An absorbing argument yields
	\begin{equation*}
		\| {v} \|_{\E_{1,\mu}}
		\leq C \cdot \| {{f}} \|_{\E_{0,\mu}} .
	\end{equation*}
	For $\eta = 0$ the operator $\mb{A}_{0} = \mb{A}$ is a isomorphism by \eqref{prop:mr}(b).
	Now the method of continuity yields that $\mb{A}_{\bar{{v}}} = \mb{A}_1$ is a isomorphism for $T_1 > 0$ such that \eqref{eq:vbar} holds. 
	Since $\| \bar{{v}} \|_{\E_{1,\mu}} < \infty$, we may choose a equidistant distribution $0 = T_0 < \dots < T_n = T'$ such that 
	\eqref{eq:vbar} holds on $(T_i, T_{i+1})$ for $i = 0,\dots,n-1$. Repeating the argument with the starting time $T_i$ and initial data ${v}_i = {v}_{i-1}(T_i)$ iteratively, yields the existence of a unique solution ${v} \in \E_{1,\mu}$ of \eqref{eq:DH} on $(0,T')$.  
	Hence, $\mb{A}_{\bar{{v}}}$ is an isomorphism, i.e. \eqref{eq:linear problem} admits a unique solution. Now \cref{thm:parameter trick} implies the claim.
%	Using the analyticity of $\textbf{H}$, the implicit function theorem implies \textcolor{black}{the claim.}\footnote{Statt dem rot markierten. Das wÃ€re die ganz kurze Variante.}
%	\textcolor{red}{
%	that there exists a $\delta > 0$ and a real analytic function
%	\begin{equation*}
%		\Phi \colon (-\delta,\delta) \to \E_1(T)
%	\end{equation*}
%	such that $\Phi(0) = \textbf{v}$ and $\tilde{H}(\lambda,\Phi(\lambda)) = 0$, i.e. $\Phi(\lambda) = \textbf{v}_\lambda$. 
%	Using the embedding
%	\begin{equation*}
%		\E_{1,\mu} \hookrightarrow \rC ((0,T),X_{\gamma,\mu})
%	\end{equation*} 
%	we obtain that
%	\begin{equation}
%		\lambda \mapsto \binom{p((1+\lambda)t,\cdot)}{u((1+\lambda)t,\cdot)}
%	\end{equation}
%	is real analytic for all fixed $t \in (0,T)$. But this implies 
%	$\textbf{v}$ is real analytic in time near $t$. Since real analyticity is a local property we conclude that $\textbf{v} \in \rC^\omega_t((0,T);X_{\gamma,\mu})$.
%	}
%	The additional spatial regularity follows from the gain of regularity of the inverse of an elliptic operator $\textbf{A}$
%	\begin{equation*}
%		\textbf{v} = \textbf{A}^{-1}( \partial_t \textbf{v} + \mb{F}(\mb{v},\mb{v}) + \mb{f}),
%	\end{equation*}
%	where $\mb{A}$, $\mb{F}$ and $\mb{f}$ are defined as in \cref{sec:wellposedness}.
	
	\medskip 
	
	Note that claim (b) is a local property in the interior. After using a suitable smooth cut-off the interface conditions plays no role anymore. 
	Now the assertion follows as in Section IV.9.4 of \cite{PS16}.  
\end{proof}

\begin{proof}[Proof of \cref{thm:reg bjs}.]
	Again we only proof the real analytic case, the smooth case follows by obvious modifications. 
	
The proof is more involved than the proof of \cref{thm:reg in bjs}. We need to apply Angenent's parameter trick not only in the time variable but also in the spatial variable.  
In contrast to the problems under consideration in \cite[Section 9.4.2]{PS16} in our case  we have different evolution equations on $\Omega_p$ and $\Omega_f$. 
Therefore, the arguments in  \cite[Section 9.4.2]{PS16} need to be modified.   
	
We fix a point $(t_0,x_0) \in (0,\varepsilon) \times \Gamma$ and choose a parametrization $\varphi \colon B_{\R^2}(0,2R) \to \R^3$ for $\Gamma$ near $x_0$.
Since $\Gamma$ is real analytic by assumption, we may choose $\varphi$ analytic. 
	We extend this map to $\Omega_- := \Omega_p$ and $\Omega_+ := \Omega_f$, respectively, by
	\begin{align*}
		\phi^-(y,z) &:= \varphi(y) + z n_\Gamma, \qquad  \text{ for } (y,z) \in B_{\R^2}(0,3R) \times (-3a,0], \\ 
		\phi^+(y,z) &:= \varphi(y) + z n_\Gamma, \qquad \text{ for } (y,z) \in B_{\R^2}(0,3R) \times [0,3a),
	\end{align*}
	for $R > 0$ and $a>0$ sufficient small. Then $\phi^{\pm}$ are diffeomorphisms and real analytic. We denote by 
	$P_\Gamma$ the tangential projection and obtain $P_\Gamma \phi^{\pm}(y,z) = \varphi(y)$. We define the truncated shift by
	\begin{equation*}
		\tau_{\xi}(y,z) = (y+\xi \chi_0(y)\zeta_0(z),z)
	\end{equation*}
	for $(y,z) \in B_{\R^2}(0,3R) \times (-3a,3a)$,
	where $\chi_0$ is a smooth cutoff on $\R^2$ with $\chi_0 \equiv 1$ if $|y| \leq R$ and $\chi \equiv 0$ if $|y| > 2R$ and $\zeta_0$ is a smooth cuoff on $\R$ with $\zeta_0 \equiv 0$ if $|z| \leq 2 a$ and $\zeta_0 \equiv 0$ if $|z| > \frac{5}{2}a$. Note that $\xi$ acts only tangentially. Further, we define
	\begin{equation*}
		\tau_{\lambda,\xi}^{\pm}(t,x)
		:= 
		\begin{cases}
			(t+\lambda \chi_{t_0}(t) ,\phi^{\pm}(\tau_{t\xi}((\phi^{\pm})^{-1}(x) ))) , &\text{ if } (t,x) \in (0,\varepsilon) \times U_{\pm}, \\
			(t + \lambda \chi_{t_0}(t),x) &\text{ if } (t,x) \in (0,\varepsilon) \times \Omega_{\pm} \setminus U_{\pm},
		\end{cases}
	\end{equation*}
	on the tubular neighbourhood $U_{\pm} := \phi^{\pm}(B_{\R^2}(0,3R) \times I_{\pm})$ of $x_0 \in \Gamma$ with $I_{-} = (3a,0]$ and $I_+ = [0,3a)$. 
	%		We denote by $
	%		\tau_{\lambda,\xi}(t,x) = (
	%		\tau^1_{\lambda,\xi}(t,x),
	%		\tau^2_{\lambda,\xi}(t,x))$ the tangential and the normal component. 
	Furthermore, we set
	\begin{equation*}
		\tau^{\pm} \colon (-r,r) \times B_{\R^3}(0,r) \to \mathrm{Diff}^\infty((0,T)\times \Omega_\pm), (\lambda,\xi) \mapsto \tau_{\lambda,\xi}^{\pm} .
	\end{equation*}
	Finally, we define the associated push forward operator by
	\begin{equation*}
		T_{\lambda,\xi} {v}
		:= 
		\binom{p\circ\tau_{\lambda,\xi}^-}{u \circ \tau_{\lambda,\xi}^+} 
	\end{equation*}
	and remark that $T_{\lambda,\xi} \colon \E_{j,\mu}(T) \to \E_{j,\mu}(T)$ for $j = 0,1$ are isomorphisms. 
	
	\medskip 
	
%	We recall our system under consideration
%	\begin{equation*}
%		\left\{
%		\begin{aligned}
%			\partial_t p - \Delta p &= f_{p}
%			&&\text{on } \Omega_p, \\
%			\partial_t u - A_m u + \mathcal{P} (u \cdot \nabla u) &= f_{f}
%			&&\text{on } \Omega_f, \\
%			\mathrm{div} u &= 0 &&\mbox{on } \Omega_f, \\
%			\sigma_{f} n_{\Gamma}&= -p n_{\Gamma} - \beta\sum_{\ell=1}^{2}(u\cdot t^{\ell}_{\Gamma})t^{\ell}_{\Gamma}&&\mbox{on}\ \Gamma, \\
%			p &= 0 &&\mbox{on}\ \Gamma_1, \\
%			u &= 0 &&\mbox{on}\ \Gamma_2, \\
%			p(0) &= p_0 &&\text{on } \Omega_p, \\
%			u(0) &= u_0 &&\text{on } \Omega_f.
%		\end{aligned}
%		\right. 
%		\label{eq:H2}
%	\end{equation*}
	Let $\tilde{{H}}$ be such that $\tilde{{H}}(p,u)^T = (f_p,f_f,p_0,u_0)$ is equivalent to the system \eqref{eq:ns-darcyp}-\eqref{eq:ns-darcyint}. Analogously to the proof of \cref{thm:reg in bjs} we define
	\begin{equation*}
		{H}\bigl(\lambda,\xi,\tbinom{p}{u}\bigr) := T_{\lambda,\xi} \tilde{{H}}\bigl( T_{\lambda,\xi}^{-1}\tbinom{p}{u}\bigr) .
	\end{equation*} 
	As in the proof of \cref{thm:reg in bjs} we obtain that $D_{{v}}{H}(0,\bar{{v}})$ is an isomorphism and that ${H} \in \rC^\omega$. 
	Further, we conclude from the implicit function theorem and 
	
	\begin{equation*}
		\binom{p|_{\Gamma}}{u|_{\Gamma}}
		\in \rH_{\mu}^{\theta,r}(0,T;\rH^{2-2\theta-\frac{1}{q},q}(\Gamma)^2)
	\end{equation*}
	together with 
	the embedding
	\begin{equation*}
		\rH_{\mu}^{\theta,r}(0,T;\rH^{2-2\theta-\frac{1}{q},q}(\Gamma)) \hookrightarrow \rC ((0,T) \times \Gamma)
	\end{equation*}
	for some $\theta \in (0,1)$, the claim. The above embedding is possible since  $\frac{1}{r}+\frac{3}{2q} < 1$. \qedhere 
%	\textcolor{red}{
%	that there exists a $\delta > 0$ and a real analytic function
%	\begin{equation*}
%		\Phi \colon (-\delta,\delta) \times B_{\R^2}(0,\delta) \to \E_1(T)
%	\end{equation*}
%	such that $\Phi(0) = \textbf{v}$ and $\tilde{\mb{H}}(\lambda,\xi,\Phi(\lambda,\xi)) = 0$.
%	Using mixed derivate theorem and the trace theorem we obtain
%	\begin{equation*}
%		\binom{p|_{\Gamma}}{u|_{\Gamma}}
%		\in \rH_{\mu}^{\theta,r}(0,T;\rH^{2-2\theta-\frac{1}{q},q}(\Gamma)^2)
%	\end{equation*}
%	for all $0 < \theta < 1-\frac{1}{2q}$.
%	For $\frac{1}{r}+\frac{3}{2q} < 1$ there exists a suitable $\theta$ such that the embedding
%	\begin{equation*}
%		\rH_{\mu}^{\theta,r}(0,T;\rH^{2-2\theta-\frac{1}{q},q}(\Gamma)) \hookrightarrow \rC ((0,T) \times \Gamma)
%	\end{equation*}
%	holds, and  
%	obtain that
%	\begin{equation*}
%		(\lambda,\xi) \mapsto \binom{p(t+\lambda \chi_{t_0}(t) ,\phi^{-}(\tau_{t\xi}((\phi^{-})^{-1}(x) )))}{u(t+\lambda \chi_{t_0}(t) ,\phi^{+}(\tau_{t\xi}((\phi^{+})^{-1}(x))))}
%	\end{equation*}
%	is real analytic for all fixed $t \in (0,T)$ and $x \in \Omega_{\pm}$, respectively. 
%	Considering $t =t_0$ and $x = x_0$ we obtain that
%	\begin{equation*}
%		(\lambda,\xi) \mapsto \binom{p(t_0+\lambda)}{x_0+t_0 \xi} = \mb{v}(t_0+\lambda,x_0+t_0 \xi)
%	\end{equation*}
%	is real analytic. This implies $\mb{v}$ is real analytic in time and tangential to the interface $\Gamma$ near $(t_0,x_0)$, since $\xi$ acts only tangentially.
%	Since real analyticity is a local property we conclude that $\textbf{v} \in \rC^\omega((0,T) \times \Gamma)$ and hence the claim.}
%	%\textcolor{red}{ToDo...}
\end{proof}	

\medskip 

{\bf Acknowledgements}
Tim Binz would like to thank DFG for support through project 538212014. Matthias Hieber acknowledges the support by DFG project FOR 5528, while Arnab Roy would like to express his gratitude to the Alexander von Humboldt-Stiftung / Foundation, Grant RYC2022-036183-I for their support.
	
	\bibliographystyle{plain}
	\bibliography{ref-porofluid}
	
\end{document}